\documentclass[12pt]{article}

\usepackage{microtype}
\usepackage{charter} 
\usepackage{scalerel,stackengine}
\stackMath
\newcommand\reallywidehat[1]{%
\savestack{\tmpbox}{\stretchto{%
  \scaleto{%
    \scalerel*[\widthof{\ensuremath{#1}}]{\kern-.6pt\bigwedge\kern-.6pt}%
    {\rule[-\textheight/2]{1ex}{\textheight}}
  }{\textheight}%
}{0.5ex}}%
\stackon[1pt]{#1}{\tmpbox}%
}

\usepackage[left=2.7cm,right=2.7cm,top=3cm,bottom=3cm]{geometry}

\usepackage{mathtools}
\usepackage{csquotes}
\usepackage{amsthm}
\usepackage{mathrsfs}
\usepackage{amsfonts}
\usepackage{array}
\usepackage{comment}
\usepackage{float}
\usepackage{url}
\usepackage{tikz}
\usepackage{soul}
\usepackage{siunitx}
\usepackage{MnSymbol}
\usepackage[utf8]{inputenc}
\usepackage{caption}
\usepackage{fancyvrb}
\usepackage[linesnumbered,ruled,vlined]{algorithm2e}
\SetKwInput{KwInput}{Input}                
\SetKwInput{KwOutput}{Output}              
\SetKwComment{Comment}{/* }{ */}

\usepackage{enumerate}

\usepackage{lipsum}
\usepackage{soul}
\usepackage{todonotes}
\usepackage{authblk}
  \usepackage[
      linktocpage=true,
      hyperfootnotes=true,
      pdftex,                %
      bookmarks         = true,
      bookmarksnumbered = true,
      pdfpagemode       = None,
      pdfstartview      = FitH,
      pdfpagelayout     = SinglePage,
      colorlinks        = true,
      linkcolor= red, 
      citecolor         = blue,
      urlcolor          = magenta,
      pdfborder         = {0 0 0}
      ]{hyperref}

\usepackage{cleveref}

\date{}

\def\Idupinfty{\mathcal{I}_{{{\operatorname{dup}}}}^\infty}
\newcommand{\seq}[1]{(#1_n)_{n\geq0}}
 
 \def\Rt{\tilde{R}}
 \def\Lt{\tilde{\Lambda}}

\usepackage{bm}

\newcommand{\Q}{\mathbb{Q}}
\newcommand{\val}{\operatorname{val}}
\newcommand{\LL}{\mathbb{L}}
\newcommand{\K}{\mathbb{K}}
\newcommand{\M}{\mathcal{M}}
\newcommand{\B}{\mathcal{B}}
\DeclareMathOperator{\Jac}{\operatorname{Jac}}
\DeclareMathOperator{\Det}{\operatorname{Det}}
\DeclareMathOperator{\Sdup}{\mathcal{S}_{{{\operatorname{dup}}}}}
\def\acfield#1{\overline{\K(#1)}}
\def\Sdupe{\mathcal{S}_{{{\operatorname{dup}}}}^\epsilon}
\DeclareMathOperator{\algdeg}{\operatorname{algdeg}}
\DeclareMathOperator{\totdeg}{\operatorname{totdeg}}
\DeclareMathOperator{\lex}{lex}

\newtheorem{thm}{Theorem}
\newtheorem{lem}[thm]{Lemma}
\newtheorem{prop}[thm]{Proposition}

\newtheorem{ex}{Example}
\newtheorem{rmk}{Remark}
\newtheorem*{ex1}{Example 1 (cont.)}
\newtheorem*{ex2}{Example 2 (cont.)}

\title{Systems of Discrete Differential Equations,\\
Constructive Algebraicity of the Solutions}
\author{\href{https://mathexp.eu/notarantonio/}{Hadrien Notarantonio}\footnote{Inria, Palaiseau (France). hadrien.notarantonio@inria.fr}, \href{https://yurkevi.ch/}{Sergey Yurkevich}\footnote{ A\&R TECH (Austria). sergey.yurkevich@univie.ac.at}}

\begin{document}

\maketitle

\begin{abstract}
In this article, we study systems of $n \geq 1$, not necessarily linear, discrete differential equations (DDEs) of order $k \geq 1$ with one catalytic variable. Such equations appear for instance in~\cite{Tutte62, BrTu64, Brown65, BMJ06, BeBM11} in the case~$n=1$ and in~\cite{BMJ06,BoBMDoPe17,BBGR19,BK20,permclass} in the case~$n>1$.
We provide a constructive and elementary proof of algebraicity of the solutions of such equations. This part of the present article can be seen as a generalization of the pioneering work by Bousquet-M\'elou and Jehanne (2006) who settled down the case $n=1$. Moreover, we obtain effective bounds for the algebraicity degrees of the solutions and provide an algorithm for computing annihilating polynomials of the algebraic series. Finally, we compare three different strategies for solving systems of DDEs in view of practical applications.
\end{abstract}

\section{Introduction}\label{sec:intro}
 \subsection{Context and motivation}
 
The equations that lie in the interest of this work are so-called \emph{discrete differential equations with one catalytic variable of fixed-point type}. They take~the~form
\begin{equation}\label{eq:DDE}
F(t, u) = f(u) + t \cdot Q(F(t, u), \Delta_a F(t, u), \dots, \Delta_a^k F(t, u), t, u),
\end{equation}
where $k \in \mathbb{N}$ is the \emph{order}, $f$ and $Q$ are polynomials, and (for some $a \in \Q$) $\Delta_a^\ell$ is the $\ell$th iteration of the~\emph{discrete derivative operator} $\Delta_a: \Q[u][[t]] \rightarrow \Q[u][[t]]$ defined~by 
\[
\Delta_a F(t,u) := \frac{F(t,u) - F(t,a)}{u-a}.
\]
Discrete differential equations (DDEs) are ubiquitous in enumerative combinatorics~\cite{Tutte62, BrTu64, Brown65, BMJ06, BeBM11}. Indeed, enumerating discrete structures usually leads to introducing the corresponding generating function, say~$G(t)$. In some of such cases, the combinatorial nature of the initial problem can be transformed directly into an algebraic or analytic question on the equation satisfied by~$G(t)$. 
However for many practical counting problems, the initial combinatorial structure is too coarse to be translated into any meaningful equation. In these cases it is often helpful to try to solve a more refined problem, introducing new structure into the initial question and consequently a new variable $u$ in the generating function: one is then led to study a DDE satisfied by some bivariate series~$F(t,u) \in \Q[u][[t]]$, even though the interest only lies in the specialization, usually $G(t) = F(t,0)$ or $G(t) = F(t,1)$. Although this idea is, of course, very classical and has been used for decades~\cite{Tutte62,BrTu64,Brown65}, the name ``catalytic'' for such a variable $u$ was introduced only relatively recently by Zeilberger~\cite{Zeilberger00} in the year 2000. 

\begin{ex} \label{ex:2-const} 
The so-called $2$-constellations are special bi-colored planar maps
~(see \cite[\S $5.3$]{BMJ06} for the definition). Let the sequence $\seq{a}$ enumerate the~$2$-constellations with~$n$ black faces; we wish to discover properties of $G(t) = \sum_{n \geq 0} a_n t^n$ (e.g. a closed-form expression for $G(t)$ or the numbers $a_n$, asymptotics of $(a_n)_{n \geq 0}$, etc). This problem is usually refined by considering the numbers~$a_{n, d}$ enumerating $2$-constellations with~$n$ black faces and outer degree~$2d$. With now more constraints on the studied enumeration (black faces and outer-degree), it is possible to show by a recursive analysis of the construction of a~$2$-constellation (see~\cite[Section~$5.3$]{BMJ06} for details) that the bivariate generating function~$F(t, u) \coloneqq \sum_{n, d\geq 0}{a_{n, d}u^dt^n}\in\Q[u][[t]]$ satisfies the~DDE~of~order~$1$
\begin{align}\label{eq:2-const}
    F(t, u) = 1 + tuF(t, u)^2 + tu \Delta_1 F(t,u). 
\end{align}
Note that~$G(t) = F(t, 1) \in\Q[[t]]$ is the generating function we were initially interested in.
The classical way of studying the series~$G(t)$ is by considering Equation~\eqref{eq:2-const}. For instance, by studying~\eqref{eq:2-const}, it can be shown using~\cite{Brown65qm}, \cite[Section~$2$]{BMJ06} or~\cite[Proposition~$2.4$]{BoChNoSa22} that~$F(t,u)$ is an algebraic function over $\Q(t,u)$ and consequently that $G(t)$ is algebraic over~$\Q(t)$. Explicitly,
\[
16t^3G(t)^2 - (8t^2 + 12t - 1)tG(t) + t(t^2 + 11t - 1) = 0.
\]
Knowing that~$R(t, z) := 16t^3z^2 - (8t^2 + 12t - 1)tz + t(t^2 + 11t - 1)\in\Q[t, z]$ annihilates~$G(t)$, it is possible to show that for~$n \geq 1$, one has the closed-form expression
\[a_n = 3\frac{2^{n-1}}{(n+2)(n+1)} \binom{2n}{n}\] whose asymptotic behavior is~$a_n \sim 3 \cdot 8^n / \sqrt{4 \pi n^{5}}$, by Stirling's formula.
\end{ex}

The algebraicity of $F(t,u)$ over~$\Q(t, u)$ in \cref{ex:2-const} is no coincidence. In their pioneering work~\cite{BMJ06}, Bousquet-Mélou and Jehanne proved (see \cite[Theorem~$3$]{BMJ06}) that the unique power series solution of a functional equation of the form~\eqref{eq:DDE} is always an algebraic function.

 In a variety of different contexts throughout combinatorics also appears their natural extension \emph{systems of DDEs} of fixed-point type that are, systems of the form
\begin{equation}\label{eqn:init_system:first}
    \begin{cases}  
F_1 = f_1(u) + t\cdot Q_1(\nabla_a^k F_1, \ldots, \nabla_a^k F_n, t, u),\\
   \indent \vdots \hspace{2cm} \vdots \\
F_n = f_n(u) + t\cdot Q_n(\nabla_a^k F_1, \ldots, \nabla_a^k F_n, t, u),
    \end{cases}
\end{equation}
where for $i=1,\dots,n$ the polynomials $f_i \in \Q[u], Q_i \in \Q[y_1,\dots,y_{n(k+1)},t,u]$ are given and we write $\nabla_a^k F \coloneqq (F, \Delta_a F, \ldots, \Delta_a^k F)$ for $\Delta_a$ as before, and the unknowns $F_i = F_i(t,u) \in \Q[u][[t]]$. As previously, $k$ is called the \emph{order} of \eqref{eqn:init_system:first}. Such systems of functional equations appear, for instance, in the enumeration of hard particles on planar maps~\cite[\S 5.4]{BMJ06}, inhomogeneous lattice paths~\cite{BK20}, certain orientations with $n$ edges \cite[\S 5]{BoBMDoPe17}, or parallelogram polyominoes~\cite[\S 7.1]{BBGR19} or even \href{https://permpal.com/}{permutation classes}~\cite{permclass}.

For any given functional equation or system of such, the results one typically wishes to obtain are for instance: a closed-form expression for the number of objects of a given size, a grasp on the asymptotic behavior, or a classification of the nature of the generating function(s) (e.g. algebraic, D-finite, etc.). It is often easier to obtain a closed-formula or an asymptotic estimate by studying a \emph{witness} of the nature of the generating function (e.g. an annihilating polynomial, a linear ODE, etc.), so a natural aim is to \emph{solve} the system, i.e. to compute such a witness. These objectives frequently yield arbitrarily difficult challenges at the intersection with enumerative combinatorics, theoretical physics~\cite{BrItPaZu78} and computational geometry~\cite{CaDeSc06}.

\paragraph{Main goals:} We want to prove in an elementary way that the components of the unique solution~$(F_1,\dots,F_n)\in\Q[u][[t]]^n$ of \eqref{eqn:init_system:first} are always algebraic power series over~$\Q(t, u)$. Moreover in the setting~$n>1$ in~\eqref{eqn:init_system:first}, we want to design, analyze and theoretically compare geometry-driven algorithms that compute an annihilating polynomial of the specialized series~$F_1(t, a)$.\\

Before providing a state-of-the-art and stating our contributions, we introduce a combinatorial example of systems of DDEs that we shall use intensively for illustrating this paper.
\begin{ex} \label{ex:eq27}
The following system of DDEs for the generating function of certain planar orientations was considered in~\cite[Eq.$(27)$]{BoBMDoPe17}:
\begin{equation}
    \begin{cases}\label{eq:eq27}
    \textbf{(E$_{\textbf{F}_1}$):} \;\;
    F_1(t, u) = 1 + t\cdot \big(u +2uF_1(t, u)^2 + 2uF_2(t, 1) + u\frac{F_1(t, u)-uF_1(t, 1)}{u - 1}\big),\\
    \textbf{(E$_{\textbf{F}_2}$):} \;\;
    F_2(t, u) = t\cdot \big(2uF_1(t, u)F_2(t, u) + uF_1(t, u) + uF_2(t, 1) + u\frac{F_2(t, u)-uF_2(t, 1) }{u - 1}\big).
    \end{cases}
\end{equation}
 We show in~\cref{sec:generic_equations_of_first_order} that $F_1(t, 1)=1+2t+10t^2+66t^3+\cdots$ is algebraic over~$\Q(t)$, and that its minimal polynomial~$64t^3z_0^3 + (48t^3 - 72t^2 + 2t)z_0^2 - (15t^3 - 9t^2 - 19t + 1)z_0 + t^3 + 27t^2 - 19t + 1$ can be computed using tools coming from elimination theory.
\end{ex}

\subsection{Previous works}
In the lines below, we present an overview of the main results regarding algebraicity of solutions of systems of DDEs. 
We start with the study of a single DDE for~$k>1$ (the case~$k=1$ is not of our interest since, as we shall see later in this article, systems of DDEs of order~$k\geq 1$ reduce to a single functional equation that contains at least~$2$ univariate series).

In the seminal work~\cite{BMJ06}, Bousquet-Mélou and Jehanne completely resolve algebraicity of solutions of DDEs~\eqref{eq:DDE}, equivalently the case $n = 1$ in \eqref{eqn:init_system:first}.
Moreover, Bousquet-Mélou and Jehanne provide systematic methods for computing an annihilating polynomial of the specialized series~$F(t, a)$. 
For proving~\cite[Theorem~$3$]{BMJ06}, the authors designed a ``nonlinear kernel method'' which allows one to prove that the unique solution of (\ref{eq:DDE}) is always an algebraic function over~$\Q(t,u)$. Significantly in practice, this approach yields an algorithm for finding an annihilating polynomial of the specialization $F(t,a)$ and of the bivariate series~$F(t,u)$. \textcolor{black}{The idea of their algebraicity proof is to reduce the resolution of the DDE to solving some system of polynomial equations which has a solution whose coordinates contains the involved specializations of~$F$. Their proof involves a symbolic deformation argument ensuring that the polynomial system which is constructed contains enough independent equations. }For efficiency considerations for the resolution of a single DDE of the form~\eqref{eq:DDE}, a recent algorithmic work~\cite{BoNoSa23} by Bostan, Safey El Din and Notarantonio targeted the intensive use of effective algebraic geometry in order to efficiently solve the underlying polynomial systems.\\
\indent Regarding systems of DDEs (i.e. the case~$n>1$), the usual strategy (e.g. \cite[Section~$11$]{BMJ06}) is to reduce a given system to a single equation and to apply the method from~\cite[Section~$2$]{BMJ06}. Nevertheless, since the reduced equation may not be of the form~(\ref{eq:DDE}) anymore, the ideas of~\cite[Section~$2$]{BMJ06} may not be applicable. 
 In the literature, there exist two methods to overcome these theoretical issues. First, a deep theorem in commutative algebra by Popescu~\cite{Popescu85a}, central in the so-called “nested Artin approximation” theory, guarantees that equations of the form~(\ref{eqn:init_system}) always admit an algebraic solution (see also~\cite[Theorem~$16$]{BoBMDoPe17} for a statement of this theorem). Note that the nested condition is automatically satisfied in this case and that the uniqueness of the solution is obvious. The approximation theorem by Popescu~\cite[Theorem~$1.4$]{Popescu85a} is a consequence of a technical regularity property~\cite[Theorem~$1.3$]{Popescu85a} in the ring of multivariate formal power series.  A drawback of using Popescu's theorem, however, is that its proof is a priori highly non-constructive and can only be applied as a ``black box'', whereas in practice one is often interested in the explicit annihilating polynomials of the solutions. Secondly, the frequent case when~(\ref{eqn:init_system}) is linear in the bivariate formal power series and their specializations was effectively solved (i.e. their proof of algebraicity yields an algorithm) in the more recent article~\cite{BK20} by Buchacher and Kauers by using a multi-dimensional kernel method. However even if their proof yields an algorithm, its efficiency was not discussed at all.  
 Note that the now common multi-dimensional kernel method appears as well in the article of the same year~\cite{ABBG20} by Asinowski, Bacher, Banderier and Gittenberger. 
 
 Before the present work, however, there was no systematic approach for dealing with systems of DDEs such as~\eqref{eq:eq27}.
It is in this context in which the present paper is situated.\looseness=-1

\subsection{Contributions} \textcolor{black}{This paper is the full version of the extended abstract~\cite{NoYu23} that was published in the proceedings of the conference \emph{Formal Power Series and Algebraic Combinatorics~$2023$}.}

The first contribution of this article is a generalization of the algebraicity result~\cite[Theorem~$3$]{BMJ06} to systems of discrete differential equations of a fixed-point type. Precisely, we prove the following theorem. Here and in the following, we denote $\K$ a field of characteristic $0$.
\begin{thm}\label{thm:main_thm}
Let $n, k\geq 1$ be integers and $f_1, \ldots, f_n\in\K[u]$, $Q_1, \ldots, Q_n\in\K[y_1, \ldots, y_{n(k+1)}, t, u]$ be polynomials. For~$a\in\K$, set~$\nabla_a^k F := (F, \Delta_a F, \ldots, \Delta_a^k F)$. Then the system of equations
\begin{align}\label{eqn:init_system}
    \begin{cases}  
\textbf{\textup{\text{(E}}}_{\textbf{\textup{\text{F}}}_\textbf{\textup{\text{1}}}} \textbf{\textup{\text{):}}}
    \;\;\; F_1 = f_1(u) + t\cdot Q_1(\nabla_a^k F_1, \ldots, \nabla_a^k F_n, t, u),\\
   \quad \vdots \hspace{3cm} \vdots \\
\textbf{\textup{\text{(E}}}_{\textbf{\textup{\text{F}}}_\textbf{{\text{n}}}} \textbf{\textup{\text{):}}} \;\;\;F_n = f_n(u) + t\cdot Q_n(\nabla_a^k F_1, \ldots, \nabla_a^k F_n, t, u)
    \end{cases}
\end{align}
admits a unique vector of solutions~$(F_1, \ldots, F_n)\in\K[u][[t]]^n$, and all its components are algebraic functions over~$\K(t, u)$.
\end{thm}

The key idea, analogous to the one in the proof of~\cite[Theorem~$3$]{BMJ06}, for proving this theorem is to define a deformation of (\ref{eqn:init_system}) that ensures the applicability of a multi-dimensional analog of the ``nonlinear kernel method''. Stated explicitly, we show in \cref{lem:det_sols} that after deforming the equations as in (\ref{eqn:deformed_system}), the polynomial in $u$ defined by
the determinant of the Jacobian matrix associated to the equations in~\eqref{eqn:init_system} (considered with respect to the $F_i$)
has exactly $nk$ solutions in an extension of the ring $\bigcup_{d \geq 1}\overline{\K}[[t^{1/d}]]$. After a process of ``duplication of variables'', we construct a zero-dimensional polynomial ideal, a non-trivial element of which must be the desired annihilating polynomial. The most technical step consists in proving the invertibility of a certain Jacobian matrix (\cref{lem:det_Jaci} and \cref{lem:Lambda}) in order to justify the zero-dimensionality.

The second contribution is the analysis of the resulting algorithm for finding annihilating polynomials of the power series $F_i(t,u)$ in \cref{thm:main_thm}. From our constructive proof we deduce a theoretical upper bound on the algebraicity degree of each $F_i$. Moreover using the radicality of the constructed~$0$-dimensional ideal, and when the field~$\mathbb{K}$ is effective (e.g.~$\mathbb{K} = \mathbb{Q}$), we also bound the arithmetic complexity of our algorithm, that is the number of operations~$(+, -, \times, \div)$ performed in~$\mathbb{K}$. Denoting by $\operatorname{totdeg}(P)$ the total degree of a multivariate polynomial $P$, we obtain the following:
\begin{thm}\label{thm:quantitative_estimates}
In the setting of \cref{thm:main_thm}, let $(F_1, \ldots, F_n)\in\K[u][[t]]^n$  be the vector of solutions and $\delta := \max(\deg(f_1),\dots,\deg(f_n), \operatorname{totdeg}(Q_1),\dots,\operatorname{totdeg}(Q_n))$. Then the algebraicity degree of each $F_i(t, u)$ over~$\mathbb{K}(t, u)$ is bounded by $        {n^{2n^2k^2}(k+1)^{n^2k^2(n+2)+n}\delta^{n^2k^2(n+2)+n}}/{(nk)!^{nk}}$.
    Moreover if~$\mathbb{K}$ is effective,
    there exists an algorithm computing an annihilating polynomial of any $F_i(t, a)$ in $O((nk\delta)^{40(n^2k+1)})$
    arithmetic operations in $\K$.
\end{thm}

\textcolor{black}{Let us emphasize that despite the desperately looking exponent~$40$ in~\cref{thm:quantitative_estimates}, 
one can still solve concrete examples from time to time as we shall see in~\cref{sec:generic_equations_of_first_order} and~\cref{sec:practical_aspects}.}

The third contribution is the full algorithmic investigation of two natural schemes for solving systems of~DDEs. For each of them, we analyze the conditions under which they might be applied, and the possible links between their respective outputs. The first algorithm consists in the classical duplication of variables argument, by following our proof of~\cref{thm:main_thm}, and then performing a brute force elimination of all irrelevant variables (\cref{prop:elimdup}). The second scheme consists in reducing the initial system of DDEs to a single polynomial functional equation where the general method of~\cite[Section~$2$]{BMJ06} and the recent algorithmic improvements made in~\cite[Section~$5$]{BoNoSa23} might apply. In this direction, we identify in~\cref{sec:elim_duplicate} sufficient conditions under which~\cite[Section~$2$]{BMJ06} and~\cite[Section~$5$]{BoNoSa23} can systematically be used. At the end of~\cref{sec:elim_duplicate}, we show that~\cref{eq:eq27} can not be solved by the state-of-the-art based on reducing a system of DDEs to a single equation (and then in applying any systematic method), while it can be solved by the systematic method that we introduce in~\cref{sec:generic_equations_of_first_order}.

\paragraph{Structure of the paper:} In \cref{sec:generic_equations_of_first_order}, we explain our method in the case of two equations of order one under the assumption that no deformation is necessary. We summarize the method in an algorithm and showcase it explicitly on~\cref{ex:eq27}. In \cref{sec:proof_thm}, we provide proofs of~\cref{thm:main_thm} and~\cref{thm:quantitative_estimates} with more details than in the extended abstract~\cite{NoYu23}.
In~\cref{sec:practical_aspects}, we study and compare the output of two natural strategies for solving systems of DDEs. 
Ultimately, we discuss some necessary future works in~\cref{sec:discussions}.

\paragraph{Notations:} Throughout this article,~$\K$ denotes a field of characteristic~$0$,~$\overline{\K}$ its algebraic closure,~$\K[[t]]$ the ring of formal power series in~$t$ with coefficients in~$\K$, and~$\overline{\K}[[t^\frac{1}{\star}]]$ the ring $\bigcup_{d\geq 1}\overline{\K}[[t^{\frac{1}{d}}]]$ of Puiseux series with rational positive exponent in the variable~$t$. Also for $n, k\geq 1$ and~$a\in\K$, we use the compact notation~$A(u)$ for any given polynomial expression of~the form~$A(F_1(t, u), F_1(t, a), \ldots, \partial_u^{k-1}F_1(t, a), \ldots, F_n(t, u), F_n(t, a), \ldots, \partial_u^{k-1}F_n(t, a), t, u)$. For an integer~$N>1$, we denote by~$\K[x_1, \ldots, x_N]$ the ring of polynomials in the variables~$x_1,\ldots, x_N$ with coefficients in~$\K$. For~$P\in\K[x_1, \ldots, x_N]$, we denote by~$\partial_{x_i}P$ the partial derivative of~$P$ with respect to the variable~$x_i$ (for~$i\in\{1, \ldots, N\}$) and by $V(P)$ its zero set in~$\overline{\K}^N$. For~$\mathcal{I}$ an ideal of~$\K[x_1, \ldots, x_N]$, we also denote by~$V(\mathcal{I})$ the zero set of~$\mathcal{I}$ in~$\overline{\K}^N$. Also, for~$\mathcal{S}$ a set of polynomials in~$\K[x_1, \ldots, x_N]$, we denote by~$V(\mathcal{S})$ the zero set of the ideal generated by~$\mathcal{S}$ in~$\K[x_1, \ldots, x_N]$. For $\mathbf{x}, \mathbf{y}$ two sets of variables, we denote~$\{\mathbf{x}\}\succ_{\lex}\{\mathbf{y}\}$ the monomial order such that~$\{\mathbf{x}\}$ (resp.~$\{\mathbf{y}\}$) is the usual degrevlex~\cite[Definition~$5$, \S~$2$, Chapter~$2$]{CoLiOSh15} order over~$\mathbf{x}$ (resp. over~$\mathbf{y}$), and such that any monomial in~$\mathbf{y}$ is lower than any monomial containing at least one variable in~$\mathbf{x}$ (by default~$\{\mathbf{x}\}$ will always denote the usual degrevlex monomial order on the variables~$\mathbf{x}$). 
We use the soft-O notation $\tilde{O}(\cdot)$ for hiding polylogarithmic factors in the argument.

\section{General strategy and application to a first example} \label{sec:generic_equations_of_first_order}

Before proving our main theorem in~\cref{sec:proof_thm}, we introduce our general method in the situation of two equations
of order~$1$. We illustrate each step with the system~\eqref{eq:eq27} from~\cref{ex:eq27}.

Starting with~\eqref{eqn:init_system} for~$n=2$ and~$k=1$, we first multiply \textbf{(E$_{\textbf{F}_1}$)} and \textbf{(E$_{\textbf{F}_2}$)} by $(u-a)^{m_1}$ and $(u-a)^{m_2}$ respectively (for $m_1, m_2\in~\mathbb{N}$) in order to obtain a system with polynomial coefficients in~$u$. By a slight abuse of notation, we shall still write \textbf{(E$_{\textbf{F}_1}$)} and \textbf{(E$_{\textbf{F}_2}$)} for those equations. Note that this system induces polynomials $E_1, E_2$ in~$\K(t)[x_1, x_2, z_0, z_1, u]$ whose specializations
to $x_1 = F_1(t, u), x_2 = F_2(t, u), z_0 = F_1(t, a), z_1 = F_2(t, a)$, denoted by~$E_1(u), E_2(u)$, are zero.

\begin{ex1}\label{ex:cont1_toy}Multiplying \textbf{(E$_{\textbf{F}_1}$)} and \textbf{(E$_{\textbf{F}_2}$)} in~\cref{ex:eq27}
by~$u-1$ gives 
\begin{equation*}
    \begin{cases}
    E_1 =  (1 - x_1)\cdot(u-1)+ t\cdot (2u^2x_1^2 - u^2z_0 + 2u^2z_1 - 2ux_1^2 + u^2 + ux_1 - 2uz_1 - u),\\
    E_2 =  x_2\cdot(1-u) + t\cdot (2u^2x_1x_2 + u^2x_1 - 2ux_1x_2 - ux_1 + ux_2 - uz_1).
    \end{cases}
\end{equation*}
Note that applying the specializations $x_1 = F_1(t, u), x_2 = F_2(t, u), z_0 = F_1(t, 1), z_1 = F_2(t, 1)$ to~$E_1$ and~$E_2$ yields the vanishing of the induced polynomial functional equations.
\end{ex1}
\noindent
In the spirit of~\cite{BMJ06}, we take the derivative of both equations with respect to the variable~$u$:
\begin{equation}\label{eqn:starting_point}
    \begin{pmatrix}
    (\partial_{x_1}E_1)(u) & (\partial_{x_2}E_1)(u)\\
    (\partial_{x_1}E_2)(u) & (\partial_{x_2}E_2)(u)
    \end{pmatrix}
    \cdot \begin{pmatrix}
    \partial_u F_1 \\
    \partial_u F_2
    \end{pmatrix} 
     +\begin{pmatrix}
    (\partial_u E_1)(u)\\
    (\partial_u E_2)(u)
    \end{pmatrix} = 0,
\end{equation}
Define $\Det:=\partial_{x_1}E_1 \cdot \partial_{x_2}E_2 - \partial_{x_1}E_2 \cdot \partial_{x_2}E_1 \in
\K(t)[x_1,x_2,z_0,z_1,u]$.
One can show that the specialization
$\Det(F_1(t, u), F_2(t, u), F_1(t, a), F_2(t, a), t, u)\in\K[[t]][[u]]$ admits either~$0, 1$ or $2$ distinct non-constant solutions in $u$ in~$\overline{\K}[[t^{\frac{1}{\star}}]]$. We assume that there exist $2$ such solutions $U_1, U_2 \in \overline{\K}[[t^{\frac{1}{\star}}]]$; we prove in~\cref{sec:proof_thm} that it is always the case up to the deformation (\ref{eqn:deformed_system}).

Exploiting the common idea to~\cite[Proof of~Theorem~$3.2$]{ABBG20} and~\cite[Section~$3$]{BK20}, we define the vector~$v := (\partial_{x_1}E_2, \;\; - \partial_{x_1}E_1) \in \K(t)[x_1,x_2,z_0,z_1,u]^2$ and plug $U_1$ for $u$ into~$v$ and~\eqref{eqn:starting_point}. Note that~$v$ is an element of the left-kernel of the square matrix in~\eqref{eqn:starting_point} $\bmod$ $\Det(x_1, x_2, z_0, z_1, u)$. After multiplication of \eqref{eqn:starting_point} by~$v$ on the left, we find a new polynomial relation relating the series $F_1(t, U_i), F_2(t, U_i), F_1(t, a), F_2(t, a), t$ and~$U_i$, namely $ \partial_{x_1} E_1 \cdot\partial_{u}E_2 - \partial_{x_1}E_2\cdot \partial_{u}E_1 = 0$ when evaluated at $x_1 = F_1(t, U_i), x_2 = F_2(t, U_i), z_0=F_1(t, a), z_1 = F_2(t, a), u = U_i$. We denote~$P:=  \partial_{x_1} E_1 \cdot\partial_{u}E_2 - \partial_{x_1}E_2\cdot \partial_{u}E_1 \in \K(t)[x_1, x_2, z_0, z_1, u]$ this new polynomial.

    Note that $P$ is the determinant of the matrix $\begin{pmatrix} \partial_{x_1}E_1 & \partial_{x_1} E_2 \\ \partial_{u}E_1 & \partial_{u} E_2\end{pmatrix}$, which is not a coincidence, as we will see in the next section.

We define the polynomial
system~$\mathcal{S} := (E_1, E_2, \Det, P)\in\K(t)[x_1, x_2, z_0, z_1, u]^4$. It admits the relevant solutions $(F_1(t, U_i), F_2(t, U_i), F_1(t, a), F_2(t, a), U_i)\in\overline{\mathbb{K}}[[t^{\frac{1}{\star}}]]^5$, for $i \in\{1, 2\}$.

\begin{ex1}Continuing \cref{ex:eq27}, we find 
\begin{equation*}
    \begin{cases}
    \Det = (4tu^2x_1 - 4tux_1 + tu - u + 1)(2tu^2x_1 - 2tux_1 + tu - u + 1),\\
    P = -2tx_1x_2 - tx_1 + tx_2 - tz_1 - x_2 + P_1\cdot u
     + P_2\cdot u^2 + P_3\cdot u^3, 
    \end{cases}
\end{equation*}
where $P_1, P_2, P_3$ are explicit (but relatively big) polynomials in~$\mathbb{Q}[x_1, x_2, z_0, z_1, t]$.
\end{ex1}

Applying in spirit the steps of~\cite[Section~$2$]{BMJ06}, we define for $i \in \{0,1\}$ the polynomial systems $\mathcal{S}_i := \mathcal{S}(x_{2i+1},\; x_{2i+2},  \;z_0, \;z_1, \;u_{i+1})$ by ``duplicating'' variables. Now observe that as~$U_1, U_2$ are assumed to be distinct, there exists a Puiseux series~$m(t)\in\K[[t^{\frac{1}{\star}}]]$ solution in~$m$ (where~$m$ is a new variable)
of the equation~$m\cdot (U_1-U_2)-1=0$. Introducing this new variable~$m$ allows us to encode in a polynomial equation the inequation~$U_1\neq U_2$, namely we will also consider $m\cdot(u_1 - u_2) - 1$.

Now in the case where the ideal $\mathcal{S}_{{{\operatorname{dup}}}}:=\langle \mathcal{S}_0, \mathcal{S}_1, m\cdot(u_1 - u_2) - 1\rangle $ has dimension~$0$ over~$\K(t)$, that is if the polynomial system 
\begin{align*}
     \{E_1(x_{1}, x_{2}, z_0, z_1, u_{1})&=0,
   E_2(x_{1}, x_{2}, z_0, z_1, u_{1})=0,\\
    \Det(x_{1}, x_{2}, z_0, z_1, u_{1})&=0,
    P(x_{1}, x_{2}, z_0, z_1, u_{1})=0,\\
    E_1(x_{3}, x_{4}, z_0, z_1, u_{2})&=0,
   E_2(x_{3}, x_{4}, z_0, z_1, u_{2})=0,\\
    \Det(x_{3}, x_{4}, z_0, z_1, u_{2})&=0,
    P(x_{3}, x_{4}, z_0, z_1, u_{2})=0,
    m\cdot (u_1-u_2)-1=0\}
\end{align*}
has finitely many solutions in~$\overline{\K}[[t^{\frac{1}{\star}}]]$,
finding an annihilating polynomial of $F_1(t,a)$ is done by computing a nonzero element of $\langle \mathcal{S}_0, \mathcal{S}_1, m\cdot(u_1 - u_2) -1\rangle  \cap \K[z_0, t]$. In general, an ideal over $\K(t)$ is \textit{zero-dimensional} if the underlying system of polynomial equations has finitely many solutions~\cite[Thm.$6$, \S$3$, Chap.$5$]{CoLiOSh15}.
\begin{ex1}Continuing Example~\ref{ex:eq27}, we compute\footnote{All computations in this paper have been performed in Maple using the C library msolve~\cite{msolve}.} a generator of the polynomial ideal
  $\langle \mathcal{S}_0, \mathcal{S}_1, m\cdot(u_1 - u_2) -1\rangle \cap \Q[z_0, t]$.
  It has degree $13$ in $z_0$ and~$14$ in $t$. In particular, it contains among its factors
  the minimal polynomial of $F_1(t, 1)$ given by 
  $64t^3z_0^3 + (48t^3 - 72t^2 + 2t)z_0^2 - (15t^3 - 9t^2 - 19t + 1)z_0 + t^3 + 27t^2 - 19t + 1$.
\end{ex1}

We summarize the algorithm described above in the compact form given by Algorithm~\ref{algo:algo1}.~\begin{algorithm}[!ht]
\DontPrintSemicolon
  \KwInput{A system of two DDEs \textbf{(E$_{\textbf{F}_1}$)}, \textbf{(E$_{\textbf{F}_2}$)} of order~$1$, with~$\mathcal{S}_{{{\operatorname{dup}}}}$ of dimension~$0$ over~$\K(t)$.}
  \KwOutput{A nonzero $R\in\K[z_0, t]$ annihilating $F_1(t, a)$.}
  Replace \textbf{(E$_{\textbf{F}_1}$)} and \textbf{(E$_{\textbf{F}_2}$)} by their respective numerators and denote by $E_1$ and $E_2$ the associated polynomials in~$\K(t)[x_1, x_2, z_0, z_1, u]$.\;
  Compute $\Det:=\partial_{x_1}E_1 \cdot \partial_{x_2}E_2 - \partial_{x_1}E_2 \cdot \partial_{x_2}E_1$ and $P := \partial_{x_1}E_1 \cdot \partial_u E_2 - \partial_{x_1}E_2\cdot \partial_u E_1 $.\;
  Set $\mathcal{S} := (E_1, E_2, \Det, P)\subset\K(t)[x_1, x_2, z_0, z_1, u]$.\;
  For $0\leq i \leq 1$, define $\mathcal{S}_i := \mathcal{S}(x_{2i+1},\; x_{2i+2},  \;z_0, \;z_1, \;u_{i+1})$.\;
  Define $\Sdup := \langle \mathcal{S}_0, \mathcal{S}_1, m\cdot(u_1 - u_
  2) -1\rangle\subset\K(t)[m, x_1, x_2, x_3, x_4, z_0, z_1, u_1, u_2]$,\;
  \textbf{Return} a nonzero element of 
  $\Sdup \cap\; \K[z_0, t]$.
\caption{\label{algo:algo1}  Solving systems of two discrete differential equations of order~$1$.}
\end{algorithm}

We remark that if the same strategy as above is applied in the case of a single equation of first order of the form
$F_1 = f(u) + t\cdot Q_1(F_1, \Delta_a F_1,  t, u)$, the presented method simplifies to the  classical method in~\cite{BMJ06} of Bousquet-Mélou and Jehanne which relies on studying the ideal $\langle E_1, \partial_{x_1} E_1, \partial_u E_1 \rangle$. Stated explicitly, $\partial_{x_1} E_1$ plays the role of $\Det$ and $\partial_uE_1$ plays the role of $P$. 

\section{Proofs of \cref{thm:main_thm} and \cref{thm:quantitative_estimates}} \label{sec:proof_thm}

\subsection{Proof of \cref{thm:main_thm}}\label{subsec:proof_main_thm}
As explained before, the statement and proof can be seen as a generalization of~\cite[Theorem~3]{BMJ06} and \cite[Theorem~2]{BK20}, so several steps are done analogously. Without loss of generality we assume that~$a=0$ and set $\Delta := \Delta_0$ and~$\nabla := \nabla_0$.

Denote by $m_1, \ldots, m_n$ the least positive integers greater than or equal to~$k$ such that multiplying~\textbf{($\textbf{E}_{\textbf{F}_{\textit{\textbf{i}}}}$)} in \eqref{eqn:init_system} by $u^{m_i}$ gives a polynomial equation; in other words, the multiplication by $u^{m_i}$ clears the denominators introduced by the application of $\Delta$. Set $\beta := \lfloor 2M/k \rfloor $ and $\alpha:= 3n^2k\cdot (\beta + 1) + 3nM$, where $M := m_1 + \cdots + m_n$. 
Let $\epsilon$ be a new variable, $\LL := \K(\epsilon)$, and let~$(\gamma_{i, j})_{1\leq i, j\leq n}$ be defined by $\gamma_{i, i} = i^k$ and $\gamma_{i, j} = t^\beta$ for $i\neq j$. Then, consider the following system which is a symbolic deformation of~\eqref{eqn:init_system} with respect to the deformation parameter~$\epsilon$:
\begin{equation}\label{eqn:deformed_system}
  \small  \begin{cases}
\textbf{\textup{\text{(E}}}_{\textbf{\textup{\text{G}}}_\textbf{\textup{\text{1}}}} \textbf{\textup{\text{):}}} \;\;\;G_1 = f_1(u) + t^{\alpha}\cdot Q_1(\nabla^kG_1, \nabla^kG_2, \ldots, \nabla^kG_n, t^{\alpha}, u)
    + t\cdot \epsilon^{k}\cdot \sum_{i = 1}^n \gamma_{1, i} \cdot \Delta^kG_i,\\
     \qquad \qquad \;\;\;\; \vdots \hfill \vdots \;\;\;\;\;\;\;\;\;\; \\
\textbf{\textup{\text{(E}}}_{\textbf{\textup{\text{G}}}_\textbf{{\text{\textit{n}}}}} \textbf{\textup{\text{):}}} \;\;\;G_n = f_n(u) + t^{\alpha}\cdot Q_n(\nabla^kG_1, \nabla^kG_2, \ldots, \nabla^kG_n, t^{\alpha}, u)
    + t\cdot \epsilon^{k}\cdot \sum_{i = 1}^n \gamma_{n, i}\cdot\Delta^kG_i.
    \end{cases}
\end{equation}
The fixed point nature of the above equations still implies the existence of a unique solution $(G_1, \ldots, G_n)\in\LL[u][[t]]^n$ (it can be seen by extracting the coefficient of~$t^m$ of each~$G_1, \ldots, G_n$).
Remark that the equalities $F_i(t^{\alpha}, u) = G_i(t, u, 0)$ relate the formal power series
solutions of~\eqref{eqn:init_system} and of~\eqref{eqn:deformed_system}.
Hence, showing that each $G_i$ is algebraic over $\LL(t, u)$ is enough to prove \cref{thm:main_thm}. Moreover, as we will see later, the algebraicity of each $G_i$ follows from the algebraicity of the series~$G_1(0), \ldots, \partial_u^{k-1}G_1(0), \ldots, G_n(0), \ldots, \partial_u^{k-1}G_n(0)$. Here, and in what follows, we shall use the short notations \[\partial_0G_i(u)\equiv\left(G_i(u), G_i(0), \partial_u G_i(0), \ldots, \partial_u^{k-1}G_i(0)\right)\] and $A(u)\equiv A(\partial_0G_1, \ldots, \partial_0G_n, t, u)$ for any polynomial $A\in\LL[X_1, \ldots, X_n, t, u]$ with $X_j := x_j, z_{k(j-1)}, z_{k(j-1)+1}, \ldots, z_{kj-1}$. Note that in the case~$n=1$, this notation implies that for any~$0\leq i \leq k-1$, the variable $z_i$ stands for~$\partial_u^iF_1(t, a)$.

Let us define $Y_{i, 0} := x_i$ and $Y_{i, j} := ({x_i - z_{k(i-1)} - \cdots - \frac{u^{j-1}}{(j-1)!}z_{k(i-1)+j-1})/u^j}$ for $1\leq i \leq n$ and~$1\leq j \leq k$.
With these definitions, multiplying each~\textbf{($\textbf{E}_{\textbf{F}_{\textit{\textbf{i}}}}$)} in \eqref{eqn:init_system} by $u^{m_i}$ and substituting the  series~$G_i$'s and their specializations by their associated variables yields the following system of polynomial equations
\begin{equation}\label{eq:E_i's}
 \footnotesize   \begin{cases}
    E_1 := u^{m_1}\cdot (f_1(u) - x_1 + t^{\alpha}\cdot Q_1(Y_{1, 0}, \ldots, Y_{1, k},Y_{2,0}, \ldots, Y_{n, k}, t^{\alpha}, u) 
    + t\cdot \epsilon^{k}\cdot\sum_{i = 1}^n \gamma_{1, i}\cdot Y_{i, k}) = 0 ,\\
    \indent \;\vdots \hfill \vdots\;\;\;\; \\
    E_n :=  u^{m_n}\cdot (f_n(u) - x_n + t^{\alpha}\cdot Q_n(Y_{1, 0}, \ldots, Y_{1, k},Y_{2,0}, \ldots, Y_{n, k}, t^{\alpha}, u)
    + t\cdot \epsilon^{k}\cdot\sum_{i = 1}^n \gamma_{n, i}\cdot Y_{i, k}) = 0.
    \end{cases}
\end{equation}
Like in (\ref{eqn:starting_point}), we take the derivative with respect to the variable~$u$ of these equations and find
\begin{equation}\label{eqn:En_diff_u}
    \begin{pmatrix}
    (\partial_{x_1}E_1)(u) & \dots & (\partial_{x_n}E_1)(u)\\
    \vdots & \ddots & \vdots \\ 
    (\partial_{x_1}E_n)(u) & \dots  & (\partial_{x_n}E_n)(u)
    \end{pmatrix}
    \cdot \begin{pmatrix}
    \partial_u G_1 \\
    \vdots \\
    \partial_u G_n
    \end{pmatrix} 
    + \begin{pmatrix}
    (\partial_u E_1)(u)\\
    \vdots\\
    (\partial_u E_n)(u)  
    \end{pmatrix} =0. 
\end{equation}
Let $\Det \in \LL[X_1,\dots,X_n,t][u]$ be the determinant of the square matrix $(\partial_{x_j}E_i)_{1\leq i,j \leq n}$. The following lemma gives the number of distinct relevant solutions in~$u$ to the equation~$\Det(u) = 0$.
\begin{lem}\label{lem:det_sols}
    $\Det(u) = 0$ admits exactly $nk$ distinct nonzero solutions
    $U_1, \ldots, U_{nk}\in \overline{\LL}[[t^{\frac{1}{\star}}]]$. 
\end{lem}
\begin{proof} Note that we have
\begin{equation*}
   \Det(u) = \det\begin{pmatrix}
    -u^{m_1} + t\epsilon^{k} \gamma_{1,1} u^{m_1-k} &  \cdots & t\epsilon^{k}\gamma_{1, n}u^{m_1-k}\\
    \vdots & \ddots & \vdots \\
    t\epsilon^{k}\gamma_{n, 1}u^{m_n-k}  &\cdots & -u^{m_n} + t\epsilon^{k}\gamma_{n, n}u^{m_n-k}
    \end{pmatrix}
     + O(t^\alpha u^{M-nk}).
\end{equation*}
For every $i$ we first divide the $i$th row by~$u^{m_i-k}$. 
Then, using the definition of $\gamma_{i,j}$ and $\alpha,\beta \geq n$, we see that the matrix above becomes diagonal mod $t^{n+1}$ and its determinant mod $t^{n+1}$ simplifies to $u^{M-nk}\cdot\prod_{j=1}^n (-u^k + t \epsilon^k j^k) \bmod t^{n+1}$. Hence, computing the first terms of a solution in~$u$ by using Newton polygons, we find $nk$ distinct nonzero solutions in~$u$ to the equation~$\Det(u)=0$. We denote these solutions by $U_1, \ldots, U_{nk}\in\overline{\LL}[[t^{\frac{1}{\star}}]]$. Their first terms are given by
$\zeta^\ell \cdot t^{\frac{1}{k}}\cdot \epsilon, \ldots, \zeta^\ell \cdot n\cdot t^{\frac{1}{k}}\cdot \epsilon
\in\overline{\LL}[[t^{\frac{1}{\star}}]]$, for $\zeta$
a $k$-primitive root of unity and for all $1\leq \ell\leq k$. Finally, note that the constant coefficient in $t$ of $\prod_{j=1}^n (-u^k + t \epsilon^k j^k)$ has degree $nk$ so by \cite[Theorem 2]{BMJ06} there cannot be more than $nk$ solutions to $\Det(u)=0$ in $\overline{\LL}[[t^{\frac{1}{\star}}]]\setminus\{0\}$.
\end{proof} 

Now, let $P$ be the determinant of the square matrix $(\partial_{x_j}E_i)_{1\leq i,j \leq n}$ where the last column $(\partial_{x_n}E_1,\dots,\partial_{x_n}E_n)$ is replaced by $(\partial_{u}E_1,\dots,\partial_{u}E_n)$, that is
 \begin{equation*}
     P:= 
        \det 
    \begin{pmatrix}
    \partial_{x_1}E_1 & \dots & \partial_{x_{n-1}}E_{1} & \partial_{u}E_1\\
    \vdots & \ddots & \vdots & \vdots \\
    \partial_{x_1}E_n & \dots & \partial_{x_{n-1}}E_{n} & \partial_{u}E_n\\
    \end{pmatrix}.
 \end{equation*}
Clearly, if $\Det(u)=0$ then \eqref{eqn:En_diff_u} implies $P(u)=0$, thus $P(u)$ vanishes at the roots $U_1,\dots,U_{nk}$ in \cref{lem:det_sols}. Hence, defining the polynomial system $\mathcal{S}$ in~$\LL[t][X_1, \ldots, X_n, u]$ given by the vanishing of the set of polynomials~$(E_1, \ldots, E_n, \Det, P)$, we see that $\mathcal{S}$ is a system with exactly~$n+2$ equations in the~$nk+n+1$ variables given by $z_0, \ldots, z_{nk-1}, x_1, \ldots, x_n, u$ (here $t$ and $\epsilon$ are parameters). 
We introduce the \emph{duplicated system} $\Sdupe := (\mathcal{S}_1, \ldots, \mathcal{S}_{nk})$, defined in~$\mathbb{L}(t)[x_1, \ldots, x_{n^2k},z_0,\dots,z_{nk-1}, u_1, \ldots, u_{nk}]$ after duplicating~$nk$ times each of the variables $x_i$'s, $u_i$'s and after duplicating~$nk$ times the initial polynomial system~$\mathcal{S}$: all in all, we perform in spirit step~$4$ of~Algorithm~\ref{algo:algo1}.
This system is built from $nk(n+2)$ equations and~$nk(n+2)$ variables.

The following lemma is proven in~\cite[Lemma~2.10]{BoChNoSa22} as a consequence of Hilbert's Nullstellensatz and of the Jacobian criterion \cite[Theorem~16.19]{Eisenbud95}. Recall that for an integer~$N>1$ and for some polynomial~$g\in\K[x_1, \ldots, x_N]$ and an ideal~$\mathcal{I}\subset\K[x_1, \ldots, x_N]$, the saturation of~$\mathcal{I}$ by~$g$ (also called saturated ideal) is defined by~$\mathcal{I}:g^\infty:=\{ f\in\K[x_1, \ldots, x_N] | \exists s\in\mathbb{N} \text{ s.t. } g^s\cdot f\in\mathcal{I}\}$. In practice, if~$\{h_1, \ldots, h_r\}$ is a generating set of~$\mathcal{I}$, then a generating set of~$\mathcal{I}:g^\infty$ is obtained by computing a generating set of~$\langle h_1, \ldots, h_r, m\cdot g-1\rangle\cap\K[x_1, \ldots, x_N]$, where~$m$ is an extra variable which encodes symbolically the set of points where~$g\neq0$ (see~\cite[p.~$205$, Thm.~$14$]{CoLiOSh15}). Also recall that an ideal $\mathcal{I}$ is called \emph{radical} if $\mathcal{I} = \sqrt{\mathcal{I}} \coloneqq \{ f : f^n \in \mathcal{I} \text{ for } n >0 \}$.
\begin{lem} \label{lem:jac_ideal}
    Assume that the Jacobian matrix $\Jac_{\Sdupe}$ of~$\Sdupe$, considered with respect to the variables 
    $x_1, \ldots, x_{n}, u_1, \ldots, x_{n^2k-n+1}, \ldots, x_{n^2k}, u_{nk},
    z_0, \ldots, z_{nk-1}$, is invertible at the point
    \begin{align*}
        \mathcal{P} = 
        (G_1&(U_1), \ldots, G_n(U_1), U_1, \ldots, G_1(U_{nk}), \ldots, G_n(U_{nk}), U_{nk}, {G}_1(0), \ldots, \partial_u^{k-1}{G}_1(0), \ldots, \\
        &{G}_n(0), \ldots, \partial_u^{k-1}{G}_n(0))
        \in\overline{\LL}[[t^{\frac{1}{\star}}]]^{nk(n+1)}\times\LL[[t]]^{nk}.
    \end{align*}
    Denote~$\mathcal{I}_{{{\operatorname{dup}}}}^\epsilon$ the ideal of~$\LL(t)[x_1, \ldots, x_{n}, u_1, \ldots, x_{n^2k-n+1}, \ldots, x_{n^2k}, u_{nk},
    z_0, \ldots, z_{nk-1}]$ generated by~$\Sdupe$. Then the saturated ideal $\mathcal{I}_{{{\operatorname{dup}}}}^\epsilon: \det(\Jac_{\Sdupe})^\infty$ is zero-dimensional and radical over $\LL(t)$. Moreover, $\mathcal{P}$ lies in the zero set of $\mathcal{I}_{{{\operatorname{dup}}}}^\epsilon: \det(\Jac_{\Sdupe})^\infty$.
\end{lem}

Therefore, in order to conclude the algebraicity of $G_i(0), \ldots, \partial_u^{k-1}G_i(0)$ over $\LL(t)$ for all $1\leq i \leq n$, it is enough to justify that $\Jac_{\Sdupe}$ is invertible at $\mathcal{P}$. 
The idea for proving~$\det(\Jac_{\Sdupe})(\mathcal{P}) \neq 0$, analogous to the proof of \cite[Theorem~$3$]{BMJ06}, is to show first that $\Jac_{\Sdupe}(\mathcal{P})$ can be rewritten as a block triangular matrix. We will then show that the diagonal blocks are invertible by carefully analyzing the lowest valuation in~$t$ of their associated determinants.

If $A \in\LL[t][X_1, \ldots, X_n, u]$,
we shall define its ``$i$th duplicated polynomial'' by \[A^{(i)} := A(X_{n(i-1)+1}, \ldots, X_{ni}, u_i).\] Then the Jacobian matrix $\Jac_{\Sdupe}(\mathcal{P})$ has the shape
\begin{align*}
    \Jac_{\Sdup}(\mathcal{P}) = \begin{pmatrix}
    A_1 &        & 0  & B_1\\
        & \ddots &    & \vdots \\
    0   &        & A_{nk} & B_{nk}\\
    \end{pmatrix} \in \overline{\LL}[[t^{\frac{1}{\star}}]]^{nk(n+2) \times nk(n+2)},
\end{align*}
where for $i = 1,\dots,nk$ the matrices\footnote{In these matrices, we emphasize that notations
like~$\partial_{x_1}E_1^{(i)}(U_i)$ are compact forms for the specializations of the duplicated polynomial $\partial_{x_1}E_1^{(i)}$
to the values~$x_{(i-1)n+1} = F_1(t, U_i(t)), \ldots, x_{in} = F_n(t, U_i(t)), u_i = U_i(t), z_{(j-1)k+\ell}=(\partial_u^\ell F_j)(t, a)$.} $A_i \in  \overline{\LL}[[t^{\frac{1}{\star}}]]^{(n+2) \times (n+1)}$ and $B_i \in  \overline{\LL}[[t^{\frac{1}{\star}}]]^{(n+2) \times nk}$ are:

\begin{equation*}
 A_i := \begin{pmatrix}
    \partial_{x_1}E_1^{(i)}(U_i) & \dots & \partial_{x_n}E_1^{(i)}(U_i)  & \partial_{u_i}E_1^{(i)}(U_i)  \\
    \vdots & \ddots & \vdots & \vdots \\
    \partial_{x_1}E_n^{(i)}(U_i)  & \dots & \partial_{x_n}E_n^{(i)}(U_i)  & \partial_{u_i}E_n^{(i)}(U_i)  \\
     \partial_{x_1}\Det^{(i)}(U_i)  & \dots & \partial_{x_n}\Det^{(i)}(U_i)  & \partial_{u_i}\Det^{(i)}(U_i)  \\
    \partial_{x_1}P^{(i)}(U_i)  & \dots & \partial_{x_n}P^{(i)}(U_i)  & \partial_{u_i}P^{(i)}(U_i)  
    \end{pmatrix}, \end{equation*}
    \begin{equation*}
        B_i :=
    \begin{pmatrix}
    \partial_{z_0}E_1^{(i)}(U_i)  & \dots & \partial_{z_{nk-1}}E_1^{(i)}(U_i) \\
    \vdots & \ddots & \vdots \\
    \partial_{z_0}E_n^{(i)}(U_i) & \dots & \partial_{z_{nk-1}}E_n^{(i)}(U_i) \\
    \partial_{z_0}\Det^{(i)}(U_i)  & \dots & \partial_{z_{nk-1}}\Det^{(i)}(U_i) \\
    \partial_{z_0}P^{(i)}(U_i)  & \dots & \partial_{z_{nk-1}}P^{(i)}(U_i)
    \end{pmatrix}.
\end{equation*}\vspace{-0.3cm}

Using $\Det(U_i) = 0$ and (\ref{eqn:En_diff_u}), we see that the first $n \times (n+1)$ minor of each $A_i$ has rank at most $n-1$. Hence, after performing operations on the first $n$ rows, we can transform the $n$
th row of $A_i$ into the zero vector. It follows that after the suitable transformation and a permutation of rows, $\Jac_{\Sdupe}(\mathcal{P})$ can be rewritten as a block triangular matrix. To give the precise form of the determinant of $\Jac_{\Sdupe}(\mathcal{P})$, we first define 
\begin{equation} \label{eq:Rdef}
\small
  R := \det    \begin{pmatrix}
   \partial_{x_1}E_1^{(i)}(U_i) & \dots & \partial_{x_{n-1}}E_1^{(i)}(U_i) & y_1 \\
    \vdots & \ddots & \vdots  & \vdots \\
    \partial_{x_1}E_{n}^{(i)}(U_i)  & \dots & \partial_{x_{n-1}}E_{n}^{(i)}(U_i) & y_n 
    \end{pmatrix}\in\mathbb{K}[\{\partial_{x_\ell} E_j^{(i)}(U_i)\}_{1\leq j\leq n, 1\leq \ell\leq n-1}][y_1, \ldots, y_n].
\end{equation}\vspace{-0.3cm}
Then it follows that $
\det(\Jac_{\Sdupe})(\mathcal{P}) = \pm \big(\prod\limits_{i = 1}^{nk}\det(\Jac_{i}(U_i))\big)\cdot\det(\Lambda)$, where
\begin{equation*}
\small
    \Jac_i(u) := 
     \begin{pmatrix}
    \partial_{x_1}E_1^{(i)}(u) & \dots & \partial_{x_n}E_1^{(i)}(u)  & \partial_{u_i}E_1^{(i)}(u)  \\
    \vdots & \ddots & \vdots & \vdots \\
    \partial_{x_1}E_{n-1}^{(i)}(u)  & \dots & \partial_{x_n}E_{n-1}^{(i)}(u)  & \partial_{u_i}E_{n-1}^{(i)}(u)  \\
     \partial_{x_1}\Det^{(i)}(u)  & \dots & \partial_{x_n}\Det^{(i)}(u)  & \partial_{u_i}\Det^{(i)}(u)  \\
    \partial_{x_1}P^{(i)}(u)  & \dots & \partial_{x_n}P^{(i)}(u)  & \partial_{u_i}P^{(i)}(u)  
    \end{pmatrix} \in  \LL[u][[t]]^{(n+1) \times (n+1)}, \text{ \normalsize{and}}
\end{equation*}
\begin{equation} \label{eq:def_Lambda}
\small
   \Lambda := \big( R(\partial_{z_j} E_1^{(i)}(U_i), \ldots, \partial_{z_j} E_n^{(i)}(U_i))\big)_{1 \leq i, j+1 \leq nk} \in \overline{\LL}[[t^{\frac{1}{\star}}]]^{nk \times nk}.
\end{equation}
The proof that this product is nonzero is the content of \cref{lem:det_Jaci} and \cref{lem:Lambda}.
\begin{lem}\label{lem:det_Jaci}
    For each $i=1,\dots,nk$, the determinant of $\Jac_i(U_i)$ is nonzero.
\end{lem}
\begin{proof}
To prove that $\det(\Jac_{i}(U_i))\neq 0$ we will show that $\operatorname{val}_t(\det(\Jac_{i}(U_i)))<\infty$, where $\val_t$ denotes the valuation in~$t$.
The main idea here is to expand~$\det(\Jac_{i}(U_i))$ with respect to the last column and show that the least valuation comes from the product of~$\partial_{u_i}\Det^{(i)}(U_i)$ by the determinant of its associated submatrix\footnote{For~$A=(a_{i, j})_{1\leq i, j\leq n}$ some matrix of size $n\times n$, the submatrix associated to an element~$a_{i_0, j_0}$ is the $(n-1)\times(n-1)$ submatrix of~$A$ obtained by deletion of the~$i_0$th row and~$j_0$th column from $A$.}, denoted by $\mathcal{M}$. For some matrix $\mathcal{A}$ whose entries are series in $t$, we shall denote by $\val_t(\mathcal{A})$ the matrix of the valuations in~$t$ of the entries of~$\mathcal{A}$.

We shall justify that the term with lowest exponent in~$t$ in~$\det(\Jac_{i}(U_i))$ comes from the first term in the product~$\det(\mathcal{M})\cdot\partial_{u_i}\Det^{(i)}(U_i)$. By construction, the  monomials in~$\{E_i\}_{1\leq i \leq n}$ that are quadratic in the variables~$\{x_j\}_{1\leq j \leq n}$ all carry a~$t^\alpha$, so that~$\partial_{x_j}\Det^{(i)}(U_i) = O(t^\alpha)$ for all~$1\leq j\leq n$ and all~$1\leq i\leq nk$.
Thus when expending~$\det(\Jac_{i}(U_i))$ with respect to the last column of~$\Jac_{i}(U_i)$, the minors associated with~$\partial_{u_i}E_{j}^{(i)}(U_i)$ and~$\partial_{u_i}P^{(i)}(U_i)$ are in~$O(t^\alpha)$ (for~$1\leq i \leq nk$). 

Thus, it remains to prove that $\operatorname{val}_t(\partial_{u_i}\Det^{(i)}(U_i)\cdot\det(\mathcal{M})) <\alpha$; this is done in the remaining part of the proof. Note that the expression of~$\Det \bmod t^{n+1}$ in the proof of~\cref{lem:det_sols} implies that
\begin{equation} \label{eq:valDetUi}
    \val_t(\partial_{u_i}\Det^{(i)}(U_i))=(M-nk)/k+ n-1 + (k-1)/k = (M-1)/k.
\end{equation}

For $\det(\mathcal{M})$ there are two cases to treat separately: 
Either (\textbf{Case 1}) we have~$U_i^k = n^kt\epsilon^k + (\text{higher powers of}~t)$, or (\textbf{Case 2}) we have~$U_i^k = m^kt\epsilon^k + (\text{higher powers of}~t)$ for some $m <n$. The reason for this distinction is that 
\begin{equation} \label{eq:diffxiEl}
    \partial_{x_j} E_\ell(u) = 
    \begin{cases}
    -u^{m_\ell} + t\epsilon^k u^{m_\ell-k}\ell^k, \text{ if } \ell = j,\\
    t^{\beta+1}\epsilon^k u^{m_\ell-k}, \text{ else },
    \end{cases}
\end{equation}
and, therefore, in \textbf{Case 1} the $(\ell,\ell)$ entry of $\M$ always has valuation in $t$ given by $m_\ell/k$ for each $\ell = 1,\dots,n-1$, while in \textbf{Case 2} the entry of $\M$ on row and column $m$ has a valuation in $t$ that depends on $\beta$.

\subparagraph{Case~$\mathbf{1}$:} Assume that the~$U_i$ of interest satisfies~
\begin{equation}\label{lem5:case1}
U_i^k = n^kt\epsilon^k + (\text{higher powers of}~t).
\end{equation}
By definition, $\M_{\ell,j} = \partial_{x_j} E_\ell(U_i)$ for $j=1,\dots,n$ and $\ell=1,\dots,n-1$ and thus from (\ref{eq:diffxiEl}) we find 
\begin{equation} \label{eq:valM}
\val_t(\mathcal{M})_{\ell, j}  =  \begin{cases}  \frac{m_\ell}{k}  & \text{ if } \ell=j,\\
\beta+\frac{m_\ell}{k} & \text{ else}. 
    \end{cases}
\end{equation}

Moreover, we claim that the bottom right entry of $\M$ has valuation $(M-1)/k$ in $t$. To prove this we compute 
\begin{align}
    \M_{n,n} = \partial_{x_n} P^{(i)}(U_i) & = 
       \partial_{x_n} \det 
    \begin{pmatrix}
    \partial_{x_1}E_1 & \dots & \partial_{x_{n-1}}E_{1} & \partial_{u}E_1\\
    \vdots & \ddots & \vdots & \vdots \\
    \partial_{x_1}E_n & \dots & \partial_{x_{n-1}}E_{n} & \partial_{u}E_n\\
    \end{pmatrix}^{(i)}(U_i) \nonumber \\
     & = 
    \det \begin{pmatrix}
    \partial_{x_1}E_1 & \dots & \partial_{x_{n-1}}E_{1} & \partial^2_{x_n, u}E_1\\
    \vdots & \ddots & \vdots & \vdots \\
    \partial_{x_1}E_n & \dots & \partial_{x_{n-1}}E_{n} & \partial^2_{x_n, u}E_n\\
    \end{pmatrix}^{(i)}(U_i) \bmod t^\alpha, \label{eq:Mnn}
\end{align} 
since $\partial_{x_n, x_j}^2 E_{\ell} \in O(t^\alpha)$ for each $j = 1,\dots,n$ and $\ell = 1,\dots,n$. 
Regarding the last column of the above matrix, a straightforward computation yields \begin{equation}\label{eqn:cross_der}
        \partial^2_{x_n, u}E_n(u) = -m_n u^{m_n-1} + t(m_n-k) n^k u^{m_n-k-1}+O(t^\alpha).
    \end{equation}
    From this and (\ref{lem5:case1}) it follows that $\val_t((\partial^2_{x_n,u}E_\ell)^{(i)}(U_i)) = (m_n-1)/k$. Moreover, using \eqref{eq:diffxiEl} and~\eqref{eq:valM},  we see that the only way to obtain in (\ref{eq:Mnn}) a term with exponent in $t$ independent of $\beta$ is to take the product of entries of the main diagonal. The sufficiently large choice of $\beta = \lfloor 2 {M}/{k} \rfloor$ implies that $\val_t(\M_{n,n}) = \val_t((\partial_{x_n} P)^{(i)}(U_i)) = \frac{m_1}{k} + \cdots + \frac{m_{n-1}}{k} + \frac{m_n-1}{k} = (M-1)/k$, where, as before, $M$ denotes $\sum_{j=1}^n m_j$.
     
     Altogether, in this case we obtain that $\val_t(\mathcal{M})$ has the shape 
     \begin{equation*}
  \small   
  \operatorname{val}_t(\mathcal{M}) = \begin{pmatrix}
    \frac{m_1}{k} & \beta +\frac{m_1}{k}& \dots & \beta +\frac{m_1}{k} & \beta + \frac{m_1}{k} \\
    \beta + \frac{m_2}{k} & \frac{m_{2}}{k} & \dots & \beta + \frac{m_{2}}{k} & \beta + \frac{m_{2}}{k} \\
    \vdots & & \ddots & \vdots  &\vdots \\
    \beta + \frac{m_{n-1}}{k} & \beta + \frac{m_{n-1}}{k} & \dots & \frac{m_{n-1}}{k}  & \beta +\frac{m_{n-1}}{k} \\
    \star& \star & \dots &\star & (M-1)/k
    \end{pmatrix}.
  \end{equation*}
It follows that the only term in the determinant of $\mathcal{M}$ whose exponent in $t$ has no dependency on $\beta$ comes from the product of the entries located on the main diagonal of~$\mathcal{M}$, and that the choice~$\beta = \lfloor 2 {M}/{k} \rfloor$ ensures that~$\val_t(\det(\mathcal{M}))<\beta$. Thus~\eqref{eq:valDetUi} implies~$ \operatorname{val}_t(\partial_{u_i}\Det^{(i)}(U_i)\cdot\det(\mathcal{M})) = 
(M-1)/k + (2M-m_n-1)/k$. Finally, from~$\alpha=3n^2k\cdot (\lfloor 2M/k \rfloor + 1) + 3nM$, it follows that \(\operatorname{val}_t(\partial_{u_i}\Det^{(i)}(U_i)\cdot\det(\mathcal{M})) <\alpha\), as wanted.

\subparagraph{Case~$\mathbf{2}$:} The second case is similar in spirit to~\textbf{Case~1}, but because of (\ref{eq:diffxiEl}) when~$U_i$ is not of the form~\eqref{lem5:case1}, computing the valuation in~$t$ of~$\det(\mathcal{M})$ is slightly more delicate. In this case \begin{equation}\label{lem5:case2}
    U_i^k = m^k t\epsilon^k + (\text{higher powers of}~t),
\end{equation}
for some $m<n$. For the sake of better readability, we shall assume without loss of generality that $m=1$, since the argument works equally well in the general case.

It follows from the definition of $\M$ and (\ref{eq:diffxiEl}) that
   \begin{equation*}  \mathcal{M} =   \begin{pmatrix}
     -U_i^{m_1}+t\epsilon^kU_i^{m_1-k} & \dots & \epsilon^kt^{\beta+1}U_i^{m_1-k}& \epsilon^kt^{\beta+1}U_i^{m_1-k}\\
     \vdots & \ddots & \vdots  & \vdots \\
     \epsilon^kt^{\beta+1}U_i^{m_{n-1}-k}  & \dots & -U_i^{m_{n-1}}+(n-1)^kt\epsilon^kU_i^{m_{n-1}-k} & \epsilon^kt^{\beta+1}U_i^{m_{n-1}-k} \\
     \partial_{x_1}P^{(i)}(U_i)  &  \dots & \partial_{x_{n-1}}P^{(i)}(U_i)   & \partial_{x_n}P^{(i)}(U_i)  
     \end{pmatrix} \operatorname{mod} t^\alpha.
     \end{equation*}
As before in (\ref{eq:valM}), for~$1\leq j\leq n$, we have~$\val_t(\epsilon^kt^{\beta+1}U_i^{m_{j}-k}) = \beta+\frac{m_j}{k}$. Also, assuming that $m=1$, we have for all~$2\leq j\leq n-1$ that
 \[
 -U_i^{m_j}+j^kt\epsilon^kU_i^{m_j-k} = \lambda_{i, j}\cdot t^{\frac{m_j}{k}} +  (\text{higher powers of}~t),
 \]
 for some nonzero~$\lambda_{i, j}\in\overline{\K}(\epsilon)$. 
However, for $j=1$ the term $\lambda_{i,j}$ vanishes and so we shall compute the valuation in $t$ of~$-U_i^{m_1} + t\epsilon^kU_i^{m_1-k}$ in this case. From the expansion of $\Det(u)$ to higher order terms
 \begin{align*}
     \Det(u) = u^{M-nk}\cdot[\prod_{1\leq j \leq n}{(-u^k+j^kt\epsilon^k)} - t^{2(\beta+1)}\epsilon^{2k}\sum_{1\leq\ell < j \leq n}\prod_{\substack{b=1, \\b\neq \ell, b\neq j}}^n{(-u^k+b^kt\epsilon^k)}]+O(t^{3\beta}),
 \end{align*}
 we use Newton's method to find the second lowest term of $U_i$ with respect to the exponent in $t$:
 \[
 U_i = \zeta^{\ell_i} t^{\frac{1}{k}} \epsilon + \lambda_i t^{2 \beta  + \frac{1}{k}} + (\text{higher powers of}~t),
 \]
 for some $1\leq \ell_i \leq k$ and $\lambda_i \neq 0$. This implies 
 that~\begin{align*}
    \val_t(\mathcal{M}_{1,1}) = \operatorname{val}_t(-U_i^{m_1}+t\epsilon^kU_i^{m_1-k}) &= \val_t(U_i^{m_1-k}) + \val_t(-U_i^k+ t\epsilon^k) \\
     &= \frac{m_1-k}{k} + \frac{(k-1)}{k} + 2\beta+\frac{1}{k} = 2\beta+\frac{m_1}{k}.
 \end{align*}
 It remains to understand the valuation in~$t$ of the last row of $\M$, i.e. of~$\{\partial_{x_j}P^{(i)}(U_i)\}_{1\leq j\leq n}$. The same argument as in (\ref{eq:Mnn}) implies that for any $j=1,\dots,n$ we have
 \begin{equation*}
   (\partial_{x_j} P)^{(i)}(U_i)= 
    \det 
\begin{pmatrix}
-U_i^{m_1}+t\epsilon^kU_i^{m_1-k} & \dots &  \epsilon^kt^{\beta+1}U_i^{m_1-k}& \partial^2_{x_j, u}E_1^{(i)}(U_i)\\
\vdots & \ddots & \vdots & \vdots \\
 \epsilon^kt^{\beta+1}U_i^{m_n-k}& \dots &  \epsilon^kt^{\beta+1}U_i^{m_n-k} & \partial^2_{x_j, u}E_n^{(i)}(U_i)\\
\end{pmatrix} \operatorname{mod} t^\alpha.
 \end{equation*} 
Moreover, in spirit of (\ref{eqn:cross_der}) it holds because of (\ref{eq:diffxiEl}) and
(\ref{lem5:case2}) that
  \begin{align*}
     \partial^2_{u, x_j}E_\ell(U_i) = \begin{cases}
     (m_\ell-1)/k & \text{ if } j=\ell,\\
     O(t^{\beta+\frac{m_\ell-1}{k}}) & \text{else.}
     \end{cases}
 \end{align*}
 Putting everything together, it follows that 
 \begin{equation*}
 \val_t (\partial_{x_j} P^{(i)}(U_i)) =
 \val_t \det \begin{pmatrix}
t^{2\beta+\frac{m_1}{k}} & \cdots &  \cdots & t^{\beta+\frac{m_1}{k}} & t^{\beta+\frac{m_1-1}{k}}\\
t^{\beta+\frac{m_2}{k}} & t^{\frac{m_2}{k}} & t^{\beta+\frac{m_1}{k}} &  \cdots & t^{\beta+\frac{m_2-1}{k}}\\
\vdots & \ddots & \ddots & \ddots & \vdots \\
      t^{\beta+\frac{m_{j}}{k}} & \cdots & t^{\frac{m_{j}}{k}} & t^{\beta+\frac{m_{j}}{k}}  & t^{(m_j-1)/k}\\
    \vdots & \ddots & \ddots & \ddots & \vdots \\
t^{\beta+\frac{m_{n-1}}{k}} & \cdots & t^{\beta+\frac{m_{n-1}}{k}} & t^{\frac{m_{n-1}}{k}} & t^{\beta+\frac{m_{n-1}-1}{k}}\\
 t^{\beta+\frac{m_n}{k}} & \cdots & t^{\beta+\frac{m_n}{k}} &  t^{\beta+\frac{m_n}{k}} & t^{\beta+\frac{m_n-1}{k}}\\
\end{pmatrix},
 \end{equation*}
where in the matrix above the $(i_0,j_0)$ entry is $t^{\beta + m_{i_0}/k}$ if $1 \leq j_0 <n$ and $1 \leq i_0 \neq j_0 \leq n$, $t^{m_{i_0}/k}$ if $1 < i_0=j_0 < n$, $t^{2 \beta + m_1/k}$ if $i_0=j_0=1$, $t^{\beta+\frac{m_{i_0}-1}{k}}$ if $j_0=n, i_0 \neq j$, and $t^{(m_j-1)/k}$ if $j_0=n, i_0 = j$. From this we obtain that $\val_t (\partial_{x_j} P^{(i)}(U_i))$ is at least $2\beta+{(M-1)/k}$, except if $j=1$, since then it is given by the valuation of the product of the lower left entry, the upper right entry and the remaining entries on the diagonal of the matrix above. In other words:
  \begin{align*}
    \val_t (\partial_{x_j}P^{(i)}(U_i)) = 
     \begin{cases}
     \beta+(M-1)/k & \text{ if } j=1, \\
      \geq {2\beta+{(M-1)/k}} & \text{else.}
     \end{cases}
 \end{align*}
 This means that $\val_t(\mathcal{M})$ has the shape 
 \begin{equation*}
\begin{pmatrix}
2\beta+\frac{m_1}{k} & \beta+\frac{m_1}{k} & \cdots & \cdots & \beta+\frac{m_1}{k}\\
\beta+\frac{m_2}{k} & \frac{m_2}{k} & \beta+\frac{m_2}{k} &  \cdots & \beta+\frac{m_2}{k}\\
\vdots & \ddots & \ddots & \ddots & \vdots \\
    \vdots & \ddots & \ddots & \ddots & \vdots \\
      \beta+\frac{m_{n-2}}{k}& \cdots & \frac{m_{n-2}}{k} & \beta+\frac{m_{n-2}}{k} & \beta+\frac{m_{n-2}}{k} \\
\beta+\frac{m_{n-1}}{k}& \cdots & \beta+\frac{m_{n-1}}{k} & \frac{m_{n-1}}{k} &  \beta+\frac{m_{n-1}}{k}\\
 \beta+\frac{M-1}{k}&  \geq 2\beta+\frac{M-1}{k} & \cdots &\geq 2\beta+\frac{M-1}{k} &\geq 2\beta+\frac{M-1}{k}\\
\end{pmatrix}.
 \end{equation*} 
 We see that the valuation of $\det(\M)$ is $2 \beta + (2M-m_n-1)/k$, since it is given by the sum of the lower left entry, the upper right entry and the remaining entries on the diagonal of the matrix above. Note that each other combination will sum to at least $4 \beta$ and the choice $\beta = \lfloor 2M/k \rfloor$ ensures that this term is strictly larger.
 
 It remains to check that also in this case our choice for~$\alpha$ was large enough to guarantee that $\operatorname{val}_t(\partial_{u_i}\Det^{(i)}(U_i) \cdot \det(\mathcal{M})) <\alpha$. Using that~$\val_t(\det(\mathcal{M})) = 2\beta + (2M-m_n-1)/k$ and~(\ref{eq:valDetUi}), we obtain
 \[ \operatorname{val}_t(\partial_{u_i}\Det^{(i)}(U_i))+ \val_t(\det(\mathcal{M}))
  = 2\beta + \frac{3M-m_n-2}{k} < 2\beta + \frac{3M}{k}.
 \]
 As~$2\beta + \frac{3M}{k} < \alpha$, this concludes the proof of~\cref{lem:det_Jaci}.
\end{proof}

\begin{lem}\label{lem:Lambda}
    The determinant of $\Lambda$ is nonzero.
\end{lem}
Proving that $\det(\Lambda)\neq 0$ is again done by analyzing the first terms in~$t$ of $\det(\Lambda)$. More precisely, 
 we are going to show that~$\det(\Lambda)$ factors modulo $t^\alpha$ as: a product of~$U_i$, the Vandermonde determinant $\prod_{i<j} (U_i - U_j)$, and a nonzero polynomial $H(t)\in\K[t]$. 
 
 Before starting with the proof we shall prove the following simple but useful fact: 

 \begin{lem}\label{lemma:invertib_dup_matrix}
Let $\mathbb{F}$ be a field and $m\in\mathbb{N}$. Consider polynomials $A_1, \ldots, A_m\in~\mathbb{F}[x]$ and define 
  $ M \coloneqq (A_i(x_j))_{1 \leq j, i \leq m}$.
Then $\det(M)=0$ if and only if there exist $\lambda_1, \ldots, \lambda_m\in\mathbb{F}$ not all equal to~$0$ such that we have the linear combination~$\sum_{i=1}^m{\lambda_i\cdot A_i(x) = 0}$.
\end{lem}
\begin{proof}
Clearly, if the linear combination $\sum_{i=1}^m{\lambda_i\cdot A_i(x) = 0}$ exists, $M$ is singular and its determinant vanishes. For the other direction, we can assume without loss of generality that, up to permutation of columns of $M$ and linear operations on them, $\deg(A_1(x)) > \cdots > \deg(A_n(x))$\footnote{Here $\deg(0) = -\infty$ and we allow us to write $-\infty > -\infty$.}. Then, since $\det(M) = 0$, we must have that $A_i(x) = 0$ for some $i=1,\dots,n$, because the product of the diagonal elements of $M$ cannot cancel otherwise.
\end{proof}

\begin{proof}[Proof of \cref{lem:Lambda}] Recall from (\ref{eq:Rdef}) and (\ref{eq:def_Lambda}) that
\begin{equation*}
\small
   \Lambda \coloneqq \big( R(\partial_{z_j} E_1^{(i)}(U_i), \ldots, \partial_{z_j} E_n^{(i)}(U_i))\big)_{1 \leq i, j+1 \leq nk} \in \overline{\LL}[[t^{\frac{1}{\star}}]]^{nk \times nk},
\end{equation*}
where
\begin{equation*}
\small
  R \coloneqq \det    \begin{pmatrix}
   \partial_{x_1}E_1^{(i)}(U_i) & \dots & \partial_{x_{n-1}}E_1^{(i)}(U_i) & y_1 \\
    \vdots & \ddots & \vdots  & \vdots \\
    \partial_{x_1}E_{n}^{(i)}(U_i)  & \dots & \partial_{x_{n-1}}E_{n}^{(i)}(U_i) & y_n 
    \end{pmatrix}\in\mathbb{K}[\{\partial_{x_\ell} E_j^{(i)}(U_i)\}_{1\leq \ell\leq n-1, 1\leq j\leq n}][y_1, \ldots, y_n].
\end{equation*}
We shall first define and analyze the polynomial matrix $\Lt$:
\begin{equation*}
    \Lt(u_1,\dots,u_n) = \Lt \coloneqq (\Rt_j^{(i)})_{1\leq i, j+1\leq nk}\in\LL[t][u_1, \ldots, u_{nk}]^{nk\times nk}, \text{ where}
\end{equation*}
\begin{align*} 
  \Rt_j \;&:= \det    \begin{pmatrix}
   \partial_{x_1}E_1 & \dots & \partial_{x_{n-1}}E_1 & \partial_{z_j} E_1\\
    \vdots & \ddots & \vdots  & \vdots \\
    \partial_{x_1}E_{n}& \dots & \partial_{x_{n-1}}E_{n}& \partial_{z_j} E_n
\end{pmatrix} \bmod t^\alpha \in\LL[t][u].
\end{align*}
It holds that
\begin{equation} \label{eq:LtL}
 \Lt(U_1,\dots,U_{nk}) \equiv \Lambda \bmod t^\alpha. 
\end{equation}
We will prove \cref{lem:Lambda} in 3 steps corresponding to the 3 claims:
\begin{itemize}
    \item \textbf{Claim 1:} For the matrix $\Lt$ it holds that $\det(\Lt) \neq 0$.
    \item \textbf{Claim 2:} We have that $\det(\Lt) = \prod_{i = 1}^{nk}{u_i^{M - nk}}\cdot 
     \prod_{i<j}{(u_i - u_j)}\cdot H(t)$, for some nonzero polynomial $H(t) \in \LL[t]$ of degree bounded by $n^2k(\beta+1)$.
    \item \textbf{Claim 3:} It holds that $\det(\Lt(U_1,\dots,U_{nk})) \bmod t^\alpha \neq 0$, in particular, we have $\Lambda \neq 0$.
\end{itemize}

Before we start with the proofs of the three claims, we mention that it follows from the definition of $E_\ell$ and our deformation that for~$0\leq j \leq nk-1$ and~$1\leq \ell \leq n$ we have
\begin{equation} \label{eq:parzE}
 \partial_{z_j}E_{\ell}\mod t^\alpha = 
    \begin{cases}
    -\frac{t\epsilon^k\ell^k}{(j-k(\ell-1))!} u^{m_\ell+j - k(\ell-1) - k}, \text{ if } j \in \{ k\ell-k, \ldots, k\ell -1\},\\
    -\frac{t^{\beta+1}\epsilon^k}{h!} u^{m_\ell+ h - k} \text{ else, with }
    h = j \bmod k \text{ and } 0\leq h < k.
    \end{cases}
\end{equation}
We will crucially use the fact that $\partial_{z_j}E_{\ell} \bmod t^\alpha$ is divisible by $u^{m_\ell - k}$ and that the quotient is a monomial in $u$ of degree at most $k-1$.\\

  \textbf{Proof of Claim~$\mathbf{1}$:} 
Assume that~$\det(\Lt)=0$. 
By~\cref{lemma:invertib_dup_matrix} applied to~$\Lt$, there exist polynomials~$\lambda_0, \ldots, \lambda_{nk-1}\in\mathbb{L}[t]$ not all equal to~$0$ such that $\sum_{i=0}^{nk-1}\lambda_i\cdot \Rt_i=0$.
For all~$0 \leq i \leq nk-1$ the matrices associated to the~$R_i$'s share the same~$n-1$ first columns, while the last column is equal to $(\partial_{z_i}E_\ell)_{1\leq \ell \leq n}$. From~\eqref{eq:parzE} it follows that for all~$1\leq i+1, \ell \leq n$ and for all
 $0\leq j \leq k-1$, we have~$\partial_{z_{ik+j}}E_\ell = \frac{u^{j}}{j!} \partial_{z_{{ik}}}E_\ell$. Thus, using
 the multi-linearity of the determinant, it follows that 
  for all~$0 \leq i \leq n-1$ and~$0 \leq j\leq k-1$, we have~$\Rt_{ik+j} = \frac{u^j}{j!}\cdot\Rt_{ik}$. This implies that the linear combination~$\sum_{i=0}^{nk-1}\lambda_i\cdot \Rt_i(u)=0$ rewrites into the form
\begin{equation}\label{eqn:lin_relation}
    \sum_{i=0}^{n-1} \sum_{j=0}^{k-1}\lambda_{ik+j}\cdot \frac{u^{j}}{j!}\cdot \Rt_{ik} =0.
\end{equation}
Moreover, the combination of \eqref{eq:diffxiEl} and \eqref{eq:parzE} implies that all entries the $i$th row of the matrix defining $\tilde{R}_j$ are divisible by~$u^{m_i-k}$. Define 
\[
P_\ell(u) \coloneqq u^{-m_\ell+k}\sum\limits_{i=0}^{n-1} \sum\limits_{j=0}^{k-1}\lambda_{ik+j}\cdot \frac{u^{j}}{j!}\cdot\partial_{z_{ik}}E_\ell\bmod t^\alpha \quad \text{for } \ell = 1,\dots,n,
\]
so that equations \eqref{eq:diffxiEl} and \eqref{eqn:lin_relation}, as well as the multi-linearity of the determinant imply that 
\begin{equation*}\label{eqn:matrix1}
\B(u) \coloneqq \begin{pmatrix}
    -u^k + t \epsilon^k& \epsilon^k t^{\beta+1}& \cdots & \epsilon^k t^{\beta+1} & P_1\\
    \epsilon^k t^{\beta+1} & -u^k + t \epsilon^k 2^k & \cdots & \epsilon^k t^{\beta+1}& P_2 \\
    \vdots & \vdots & \ddots & \vdots& \vdots \\
    \epsilon^k t^{\beta+1}  & \epsilon^k t^{\beta+1}& \cdots & -u^k+ t\epsilon^k  (n-1)^k &  P_{n-1}\\
    \epsilon^k t^{\beta+1} & \epsilon^k t^{\beta+1}& \cdots & \epsilon^k t^{\beta+1}&  P_n
    \end{pmatrix}
\end{equation*}
is a singular polynomial matrix. Observe from \eqref{eq:parzE} that $\deg_u P_\ell < k$ for $\ell = 1,\dots,n$. Using this and $\det(\B) = 0$ we shall now show that $P_1 = \cdots = P_n = 0$.

First note that when computing the determinant of $\B$ using the Leibniz rule, we would get the product of all entries on the main diagonal and products of other terms which will be polynomials in $u$ of degree at most $(n-1)k-1$. The contribution of the main diagonal contains $u^{(n-1)k} P_n$. Clearly, this term cannot cancel with others, unless $P_n$ vanishes. 

Fix some $1 \leq \ell \leq n-1$ and define $\omega_1,\dots,\omega_k \in \overline{\K}[\epsilon][[t^{1/\star}]]$ to be the distinct solutions of $u^k = t \epsilon^k \ell^k - t^{\beta+1} \epsilon^k$. Clearly, $\det(\B(\omega_i)) = 0$ for $i=1,\dots,k$, and by definition of $\omega_i$, the first $n-1$ entries of the $\ell$th row of $\B(\omega_i)$ agree with the first $n-1$ entries of its last row. Therefore, expanding along the last column, we find that
\[
0 = \det(B(\omega_i)) = \pm P_\ell(\omega_i) \cdot \det(\B_{\ell,n}(\omega_i)),
\]
where $\B_{\ell,n}$ denotes the matrix $\B$ with the $\ell$th row and last column removed. By subtracting the $\ell$th column from all others one can easily compute $\det(\B_{\ell,n}(\omega_i))$, in particular it follows that it is nonzero. So we conclude that $P_\ell(\omega_i) = 0$ for $i=1,\dots,k$ and, then, since $\deg_u P_\ell(u) < k$, it finally follows that $P_\ell$ vanishes identically.

By~\eqref{eq:parzE} we have that $\deg_u \partial_{z_{ik}}E_\ell = m_\ell - k$, therefore, looking at the coefficient of $u^j$ of $P_\ell(u)=0$ for $j=0,\dots,k-1$, we obtain with \eqref{eq:parzE} the following linear relations for the~$\lambda_i$'s:
\begin{equation}\label{eqn:diag_key_matrix}
    \begin{pmatrix}
    1 & t^\beta & \cdots & t^\beta\\
    t^\beta & 2^k & \cdots & t^\beta\\
    \vdots & \ddots & \ddots & \vdots \\
    t^\beta & \cdots & \cdots & n^k
    \end{pmatrix}\times \begin{pmatrix}
    \lambda_{j} \\
    \lambda_{k+j}\\
    \vdots \\
    \lambda_{(n-1)k+j}
    \end{pmatrix} = 0, \;\text{ for all $0\leq j\leq k-1$}.
\end{equation}
As the square matrix in \eqref{eqn:diag_key_matrix} is invertible, we obtain that $\lambda_i = 0$ for all $i = 0, \dots, nk-1$ which contradicts our assumption. Hence we have proved that~$\det(\Lt)\neq~0$. \\

\noindent
\textbf{Proof of Claim~2:} We prove the following explicit factorization of $\det(\Lt)$:
\begin{equation} \label{eq:detLfact}
    \det(\Lt) = \prod_{i = 1}^{nk}{u_i^{M - nk}}\cdot \prod_{i<j}{(u_i - u_j)}\cdot H(t),
\end{equation}
where $H(t) \in \K[\epsilon,t]$ is nonzero and of degree in $t$ at most $n^2(\beta+1)$. 

From \eqref{eq:diffxiEl} we know that for~$0\leq j \leq nk-1$, the polynomial $\Rt_j$ is given by 
\begin{equation*}
\det
    \begin{pmatrix}
    u^{m_1} + t \epsilon^k u^{m_1 - k}& t^{b+1}u^{m_1 - k}  & \cdots & \epsilon^k t^{b+1}u^{m_1 - k}  & \partial_{z_{j}}E_1\\
    \epsilon^k t^{b+1}u^{m_2 - k}  & u^{m_2} + t \epsilon^k 2^k u^{m_2 - k}& \cdots & \epsilon^k t^{b+1}u^{m_2 - k} &\partial_{z_{j}}E_2 \\
    \vdots & \vdots & \ddots & \vdots& \vdots \\
    \epsilon^k t^{b+1}u^{m_{n-1} - k}  & \epsilon^k t^{b+1}u^{m_{n-1} - k} & \cdots & u^{m_{n-1}-k} (u^k+ t \epsilon^k (n-1)^k) & 
     \partial_{z_{j}}E_{n-1}\\
   \epsilon^k  t^{b+1}u^{m_{n} - k}  & \epsilon^k t^{b+1}u^{m_{n} - k} & \cdots & \epsilon^k t^{b+1}u^{m_{n} - k} &  \partial_{z_{j}}E_n 
    \end{pmatrix} \bmod t^\alpha.
    \end{equation*}
    It follows from \eqref{eq:parzE} that $ u^{-m_\ell+k}\deg_{z_j}E_\ell$ is a monomial of degree at most $k-1$ in $u$. Since all other terms in the $\ell$th row of $\Rt_j$ are trivially divisible by $u^{m_{\ell}-k}$, and $\alpha > n(\beta+1)$, it follows that $\Rt_j$ factors as~$u^{M-nk}$ times the determinant of the polynomial matrix
    \begin{equation}\label{matrix1}
        \begin{pmatrix}
         u^{k} + t \epsilon^k &\epsilon^k t^{b+1} & \cdots & \epsilon^k t^{b+1}  & u^{-m_1+k}\partial_{z_{j}}E_1\\
    \epsilon^k t^{b+1} & u^{k} + t \epsilon^k 2^k& \cdots & \epsilon^k t^{b+1} &u^{-m_2+k}\partial_{z_{j}}E_2 \\
    \vdots & \vdots & \ddots & \vdots& \vdots \\
    \epsilon^k t^{b+1} & \epsilon^k t^{b+1}& \cdots & u^k+ t \epsilon^k (n-1)^k & 
    u^{-m_{n-1}+k} \partial_{z_{j}}E_{n-1}\\
   \epsilon^k  t^{b+1} & \epsilon^k t^{b+1}& \cdots & \epsilon^k t^{b+1} &  u^{-m_n+k}\partial_{z_{j}}E_n 
    \end{pmatrix} \bmod t^\alpha,
    \end{equation}
    where, as before, $M \coloneqq \sum_{i=1}^n m_i$. Thus, $u^{M-nk}$ divides~$\Rt_j$. Moreover, the degree in~$u$ of the first~$n-1$ columns of~\eqref{matrix1} is upper bounded by~$k$ and the last column of this matrix has degree at most $k-1$, so it follows that~$\deg_u(\Rt_j) \leq M-1$.
    
    Let us consider $\Lt$ which is by definition the matrix $ (\Rt_j^{(i)})_{1\leq i, j+1\leq nk}\in\LL[t][u_1, \ldots, u_{nk}]^{nk\times nk}$. Since $\deg_u(\Rt_j) \leq M-1$, we also have that $\det_{u_i}(\Lt) \leq M-1$ for each $i=1,\dots,nk$. From the considerations above it also follows that $\det(\Lt)$ is divisible by $\prod_{i = 1}^{nk}{u_i^{M - nk}}$. Moreover, it is obvious that if $u_i = u_j$ for some $i\neq j$, the determinant of $\tilde{\Lambda}$ vanishes. Hence, we can also factor out the Vandermonde determinant $\prod_{i<j}{(u_i - u_j)}$ from $\det(\Lt)$. By comparing degrees in each $u_i$ it follows that the remaining factor is a constant with respect to $u_i$ and hence \eqref{eq:detLfact} holds. By \textbf{Claim 1}, $H(t)$ is nonzero, so it remains to prove that $\deg_t H(t) \leq n^2k(\beta+1)$. This follows directly from the fact that the degree in $t$ of the $\Rt_j$'s is bounded by $n(\beta+1)$ as determinants of $(n\times n)$ matrices whose polynomial coefficients have degree in~$t$ upper bounded by~$\beta+1$.

Finally for this step, using the definitions~$\alpha = 3n^2k\cdot(\beta+1)+3nM$ and~$\beta = \lfloor 2 {M}/{k} \rfloor$, we have that~$\alpha > n^2k(\beta+1)$. So that $\deg_t(\det(\Lt))<\alpha$: by~\textbf{Claim~$1$}, we obtain $\det(\Lt)\neq 0\bmod t^\alpha$.\\

\noindent
\textbf{Proof of Claim 3:}  Having established \eqref{eq:detLfact} it is now enough to show that
\begin{equation} \label{eq:valdetLt}
    \val_t(\prod_{i = 1}^{nk}{U_i^{M - nk}}\cdot \prod_{i<j}{(U_i - U_j)}\cdot H(t)) < \alpha.
\end{equation}
Recall from \cref{lem:det_sols} that $U_i = t^{1/k} \epsilon j \zeta^\ell + (\text{higher powers of}~t)$ for $i=1,\dots,nk$, $j=1,\dots,n$ and $\zeta$ a primitive $k$th root of unity. Thus, and because $\deg_t(H) \leq n^2k(\beta+1)$ by \textbf{Claim 2}, the left-hand side of (\ref{eq:valdetLt}) is bounded by~$nk\cdot \frac{M-nk}{k} + nk(nk-1)\cdot \frac{1}{k} + n^2k(\beta+1)$, which is at most~$nM+n^2k(\beta+1)$. Since $\alpha = 3n^2k\cdot (\beta+1)+3nM$, we conclude that~$\det(\Lt(U_1,\dots,U_{nk})) \bmod t^\alpha \neq 0$. Finally, together with (\ref{eq:LtL}) this implies that $\det(\Lambda)$ does not vanish as well.
 \end{proof}
 
Having now proved that $\det(\Jac_{\Sdupe}) \neq 0$ at $\mathcal{P}$, we can apply \cref{lem:jac_ideal} and obtain that the specialized series $G_i(0), \ldots, \partial_u^{k-1}G_i(0)$ are all algebraic over~$\mathbb{K}(t, \epsilon)$. The algebraicity of the complete formal power series $G_1, \ldots, G_n$ over~$\mathbb{K}(t, u, \epsilon)$ then follows again by \cite[Lemma~2.10]{BoChNoSa22} from the invertibility of the Jacobian matrix of $E_1, \ldots, E_n$ considered with respect to the variables $x_1, \ldots, x_n$ (with $t, u, z_0, \ldots, z_{nk-1}$ viewed as parameters). 
The equalities $F_i(t^{\alpha}, u) = G_i(t, u, 0)$ finally imply that $F_1, \ldots, F_n$ are also algebraic over~$\mathbb{K}(t, u)$.

\subsection{Proof of \cref{thm:quantitative_estimates}}\label{subsec:proof_complexity_result}

We prove the quantitative estimates announced in~\cref{thm:quantitative_estimates}. 
 The techniques of the proof being standard in effective algebraic geometry, we will be quite brief in the arguments. For more details on the upper bound on the algebraicity degree, we refer the reader to~\cite[Prop.~2.8]{BoChNoSa22} in the case~$n=k=1$, and to~\cite[Prop.~3]{BoNoSa23} for the more general case~$n=1$ ($k$ arbitrary). For the complexity proof, we refer the reader to~\cite[Prop.~2.9]{BoChNoSa22} and \cite[Prop.~$4$]{BoNoSa23} in the two respective cases. As in the proof of \cref{thm:main_thm}, we assume without loss of generality that~$a=0$. \\

\noindent \textbf{Algebraicity degree bound:} In order to bound the algebracity degree of $F_i(t,u)$, we will provide an upper bound on the algebraicity degree of~$G_i(t, u, \epsilon)$ over~$\K(t, u, \epsilon)$, where, as in the proof of \cref{thm:main_thm}, $(G_1,\dots,G_n)$ denotes the solution of the deformed system (\ref{eqn:deformed_system}). Following the lines of the proof of Theorem~\ref{thm:main_thm}, we shall first give explicit bounds on the total degrees of all equations in~$\Sdupe$. Note that in this part we are interested in the algebraicity degree of $G_i(t,0)$ over $\K(t,\epsilon)$, so in the computation of the total degree we take into account all the variables $u_1, \ldots, u_{nk},x_1,\dots,x_{n^2k},z_0,\dots,z_{nk-1}$ but not $t$ and $\epsilon$. In order to distinguish this restricted notion of total degree from the usual total degree denoted by $\totdeg$, we will write it~$\totdeg_{t,\epsilon}$. So for example, $\totdeg(u_1 t \epsilon) = 3$ but $\totdeg_{t,\epsilon}(u_1t\epsilon) = 1$.

Let~$\delta$ be a bound on $\totdeg_{t,\epsilon}$ of $f_1, \ldots, f_n, Q_1, \ldots, Q_n$; then the total degrees of~$E_1, \ldots, E_n$ are bounded by $\delta (k+1)$. Moreover, $\totdeg_{t,\epsilon}(\Det), \totdeg_{t,\epsilon}(P) \leq n\delta (k+1)$, since $\Det$ and $P$ are determinants of polynomial matrices of size $n\times n$  and degree at most~$\delta (k+1)$.

Recall from \cref{lem:jac_ideal} that the ideal $\mathcal{I}_{{{\operatorname{dup}}}}^\epsilon \subseteq \mathbb{K}(t, \epsilon)[x_1, \ldots, x_{n^2k}, u_1, \ldots, u_{nk}, z_0, \ldots, z_{nk-1}]$ is defined by $nk$ duplications of  the polynomials $E_1, \ldots, E_n, \operatorname{Det}, P$. By \cref{lem:jac_ideal} the saturated ideal $\mathcal{I}_{{{\operatorname{dup}}}}^\epsilon : \det(\operatorname{Jac}_{\Sdupe})^\infty$ is radical and of dimension~$0$, and this yields bounds for the algebraicity degrees of the specializations~$\{\partial_u^iG_j(0)\}_{0\leq i\leq k-1, 1\leq j\leq n}$. More precisely, applying the Heintz-Bézout theorem~\cite[Theorem~$1$]{Heintz83} in the same way as
in~\cite[Prop.~$3.1$]{BoChNoSa22},  one obtains that the algebraicity degree of any $\partial_u^iG_j(0)$ is bounded by~\[\max_{i=1,\dots,n}(\totdeg_{t,\epsilon}(E_i))^{n^2k}\cdot \operatorname{totdeg}_{t,\epsilon}(\operatorname{Det})^{nk}\cdot \totdeg_{t,\epsilon}(P)^{nk} \leq n^{2nk}\cdot (\delta (k+1))^{nk(n+2)}.\]
We denote this bound by $\gamma(n, k, \delta)$.

As in~\cite[Prop.~$4$]{BoNoSa23} there is a group action of the symmetric group $\mathfrak{S}_{nk}$ on the zero set associated to 
the ideal $\langle E_1^{(1)}, \ldots, E_1^{(nk)}, \ldots, E_n^{(1)}, \ldots, E_n^{(nk)},
    \operatorname{Det}^{(1)}, \ldots, \operatorname{Det}^{(nk)}, P^{(1)}, \ldots, P^{(nk)} \rangle$ which permutes
    each of the $nk$ duplicated blocks of coordinates.  As this action preserves the 
    $\{t, z_0, \ldots, z_{nk-1}\}$-coordinate space, one spares the cardinality of the orbits and deduces
    that an algebraicity upper bound on any of the specialized series~$\partial_u^iG_j(0)$ is also given by 
  $  {\gamma(n, k, \delta)/(nk)!}$, that is we have 
    \begin{equation}\label{eq:boundalgdegspec}\algdeg_{\K(t, \epsilon)}(\partial_u^iG_j(0)) \leq  n^{nk} \cdot (\delta(k + 1))^{nk(n+2)}/(nk)!.\end{equation}
    
    Finally, it remains to give algebraicity bounds for the full series $G_1, \ldots, G_n$.
    This is done in the same fashion as above, this time over the field extension
    \begin{equation*}
        \mathbb{K}(t, \epsilon, G_1(0), \ldots, \partial_u^{k-1}G_1(0), \ldots, G_n(0), \ldots, \partial_u^{k-1}G_n(0))/\mathbb{K}(t, \epsilon). 
    \end{equation*}
    Using (\ref{eq:boundalgdegspec}) and the multiplicativity of the degrees in field extensions yields 
    \[
    \algdeg_{\K(t, u, \epsilon)}(G_i(u)) \leq \delta^n(k+1)^n\cdot \left(\frac{n^{2nk}(\delta(k+1))^{nk(n+2)}}{(nk)!}\right)^{nk} \quad \text{ for } i=1,\dots,n,
    \]
    and, by specialization, the same bound then holds for $\algdeg_{\K(t, u)}(F_i(u))$ as well.\\
  
    \noindent
    \textbf{Complexity estimate:} Let us recall that we fix~$\mathbb{K}$ to be an effective field of characteristic~$0$ and we count the number of elementary operations
    $(+, -, \times, \div)$ in~$\mathbb{K}$. 
    For proving the arithmetic complexity estimate in~\cref{thm:quantitative_estimates}, we rely on the use of the parametric geometric resolution 
    (see~\cite{Schost03} for details). For this purpose we need:
    \begin{itemize}
        \item the number of parameters and of variables of the input equations of interest,
    \item an upper bound on the degree (with respect to all the variables and parameters) of the saturated ideal~$\mathcal{I}_{{{\operatorname{dup}}}}^\epsilon:\det(\Jac_{\Sdupe})^\infty$.
    \end{itemize}
    In our setting the parameters are~$t$ and~$\epsilon$ so we have two of them, and the variables are \[x_1, \ldots, x_{n^2k}, z_0, \ldots, z_{nk-1}, u_1, \ldots, u_{nk},\] so $nk(n+2)$ in total.
    
    For computing an upper bound $\nu(n,k,\delta)$ on the degree of the ideal~$\mathcal{I}_{{{\operatorname{dup}}}}^\epsilon:\det(\Jac_{\Sdupe})^\infty$, we again apply the Heintz-Bézout theorem. Note that this time we must take into account the total degree of the input system~$\Sdupe$ with respect to both the variables and the parameters. By the same arguments as in the first part of the proof and using the definitions of $\alpha,\beta$, we find
  
    \[
    {\nu(n, k, \delta)} = n^{2nk}(\delta\cdot(k+1) + 3n^2k\cdot(2n\delta+1)+3nM)^{nk(n+2)}.
    \]
    Without loss of generality, we only address the computation of an annihilating polynomial of $G_1(0)$ (and consequently of $F_1(0)$). The algorithm we propose consists of two steps:
    
    {\em Step~$1$:} For $\lambda$ a new variable which is a random $\K[t,\epsilon]$-linear combination of all the variables involved in~$\mathcal{S}_{{{\operatorname{dup}}}}^\epsilon$, we compute two polynomials $V(t,\epsilon,\lambda), W(t,\epsilon,\lambda) \in \K[t, \epsilon][\lambda]$, such that: $G_1(0)$ is a
    solution\footnote{More precisely: the $z_0$th coordinate of a solution of $\Sdupe$ which is not a solution of~$\det(\Jac_{\Sdupe})=0$.} of $\Sdupe$ and not a solution of~$\det(\Jac_{\Sdupe})=0$ if and only if $G_1(0) = V(t, \epsilon, \lambda_0)/\partial_\lambda W(t, \epsilon, \lambda_0)$ for $\lambda_0 \in \overline{\K(t, \epsilon)}$ a solution of $W(t, \epsilon, \lambda)=0$.
    
    {\em Step~$2$:} Compute the squarefree part~$R\in\K[t, \epsilon, z_0]$ of the resultant of~$z_0 \cdot \partial_\lambda W(t, \epsilon, \lambda) - V(t, \epsilon, \lambda)$ and~$W(t, \epsilon, \lambda)$ with respect to~$\lambda$.  
    
    By definition, $\lambda_0$ is a solution of $\partial_\lambda W(t, \epsilon, \lambda) G_1(0) - V(t, \epsilon, \lambda) = 0$ and $W(t,\epsilon,\lambda) = 0$, so $R(t,\epsilon,G_1(0)) = 0$ by the property of the resultant. Moreover, it follows from the eigenvalue theorem~\cite[Theorem~$1$]{Cox21} and the fact that the dimension of the radical ideal $\mathcal{I}_{{\operatorname{dup}}}^\epsilon:\det(\Jac_{\Sdupe})^\infty$ is zero, that $R$ is a nonzero polynomial.
    
    By~\cite[Theorem~$2$]{Schost03}, performing \textit{Step 1} can be done using~$\tilde{O}((Ln^2k+(n^2k)^4)\cdot \nu(n, k, \delta)^3)$ operations in~$\K$, where $L$ denotes the complexity of evaluating the duplicated system~$\Sdupe$ at all variables and parameters. Using the Bauer-Strassen theorem~\cite[Theorem~$1$]{BaSt83}, the cost for evaluating~$\Det$ and~$P$ is in~$O(nL')$, where~$L'$ is the cost for evaluating any of the~$E_1, \ldots, E_n$. The total degrees of the~$E_i$'s are bounded by~$d' := \delta\cdot(k+1) + 3n^2k\cdot(2n\delta+1)+3nM$, so ~$L'\in O((d')^{nk+n+3}))$, since~$nk+n+3$ is the number of variables and parameters. Duplicating~$nk$ times the initial polynomials~$E_1, \ldots, E_n, \Det, P$, it follows that~$L\in O(n^3k L')\subset O(n^3k(d')^{nk+n+3})$. To do \textit{Step~2}, we perform evaluation-interpolation on~$t$ and~$\epsilon$. This requires to use in total~$O(\nu(n, k, \delta)^2)$ distinct points from the base field~$\K$. For each specialization of $z_0 \cdot \partial_\lambda W(t, \epsilon, \lambda) - V(t, \epsilon, \lambda)$ and~$W(t, \epsilon, \lambda)$ to~$t=\theta_1 \in \K$ and~$\epsilon=\theta_2 \in \K$, we compute the squarefree part of the bivariate resultant  of~$\partial_\lambda W(\theta_1, \theta_2, \lambda)\cdot z_\ell - V(\theta_1, \theta_2, \lambda)$ and $W(\theta_1, \theta_2, \lambda)$ with respect to~$\lambda$. Using~\cite{Villard18}, this requires~$\tilde{O}((\nu(n, k, \delta)/(nk)!)^{3.63})$ operations in~$\mathbb{K}$.

    In total we obtain that the arithmetic complexity is bounded by~$(2nk\delta)^{O(n^2k)}$ operations in~$\mathbb{K}$. Deducing
    an annihilating polynomial of~$F_1(t, 0)$ from the one annihilating~$G_1(t, \epsilon=0, u=0)$ is included in the previous complexities. This proves~\cref{thm:quantitative_estimates}. \qed
    
\section{Elementary strategies for solving systems of DDEs}\label{sec:practical_aspects}

The present section aims at studying two strategies available at this time for solving a system of DDEs in \emph{practice}. Both approaches have the advantage to be much more efficient than the algorithm resulting from \cref{thm:main_thm}, however they are not guaranteed to always terminate. As before, the goal is, given a system of DDEs of the form~\eqref{eqn:init_system}, to compute a nonzero polynomial~$R\in\K[t, z_0]$ such
that~$R(t, F_1(t, a))=~0$. 

First, in~\cref{sec:duplicate}, we summarize the algorithm underlying the duplication strategy that was used in the proof of~\cref{thm:main_thm} and we analyze its arithmetic complexity \emph{in the case where no symbolic deformation is necessary} (the arithmetic cost for solving a system of DDEs when such a symbolic deformation is needed was the content of~\cref{thm:quantitative_estimates}). In~\cref{sec:elim_duplicate}, we study the practical approach introduced and used in~\cite[Section~$11$]{BMJ06}: this approach consists in \emph{reducing} the algorithmic study of a system of DDEs to the \emph{study of a single functional equation} (which may not be of a fixed-point type). Once we are left with this single equation, we apply the geometry-driven algorithm from~\cite[Section~$5$]{BoNoSa23} which avoids the duplication of the variables.

\subsection{The classical duplication of variables algorithm}\label{sec:duplicate}
\subsubsection{General strategy}\label{subsec:general_strategy_duplicate}
Consider the system of DDEs
\begin{equation}\label{eqn:initDDEsec4}
    \begin{cases}  
\textbf{\textup{\text{(E}}}_{\textbf{\textup{\text{F}}}_\textbf{\textup{\text{1}}}} \textbf{\textup{\text{):}}}
    \;\;\;F_1 = f_1(u) + t\cdot Q_1(\nabla_a^k F_1, \ldots, \nabla_a^k F_n, t, u),\\
   \indent \vdots \hfill \vdots\;\;\;\;\; \\
\textbf{\textup{\text{(E}}}_{\textbf{\textup{\text{F}}}_\textbf{{\text{n}}}} \textbf{\textup{\text{):}}}\;\;F_n = f_n(u) + t\cdot Q_n(\nabla_a^k F_1, \ldots, \nabla_a^k F_n, t, u),
    \end{cases}
\end{equation}
introduced in~\cref{thm:main_thm} and denote by~$\Sdup$ the duplicated polynomial system obtained~by 
\begin{itemize}
\item considering the polynomials $E_1, \ldots, E_n\in\K(t)[x_1, \ldots, x_n, u, z_0, \ldots, z_{nk-1}]$
  associated to the ``numerator equations''\footnote{For all~$1\leq i \leq n$,
  $E_i(u) \equiv E_i(F_1, \ldots, F_n, u, F_1(t, a), \ldots, \partial_u^{k-1}F_1(t, a), \ldots, F_n(t, a),
  \ldots, \partial_u^{k-1}F_n(t, a))$.}
  of the system of DDEs~\eqref{eqn:initDDEsec4}, as well as the polynomials
      \[
        \Det := 
        \det \begin{pmatrix}
            \partial_{x_1}E_1 & \dots & \partial_{x_n}E_1\\
            \vdots & \ddots & \vdots\\
            \partial_{x_1}E_n & \dots & \partial_{x_n}E_n\\
        \end{pmatrix} \quad
        \text{ \normalsize{and} } \quad
        P:= 
        \det 
    \begin{pmatrix}
    \partial_{x_1}E_1 & \dots & \partial_{x_{n-1}}E_{1} & \partial_{u}E_1\\
    \vdots & \ddots & \vdots & \vdots \\
    \partial_{x_1}E_n & \dots & \partial_{x_{n-1}}E_{n} & \partial_{u}E_n\\
    \end{pmatrix},\]
  \item  duplicating~$nk$ times the variables $x_1, \ldots, x_n, u$ in the polynomials~$(E_1, \ldots, E_n, \Det, P)$
    in order to obtain~$nk$ duplications
    of the polynomials~$(E_1, \ldots, E_n, \Det, P)$ in the new variables~$x_1, \ldots, x_{n^2k}$,
    $u_1, \ldots, u_{nk}$ and (unchanged variables)~$z_0, \ldots, z_{nk-1}$ (as in~\cref{ex:cont1_toy} in the case~$n=2$ and~$k=1$).
\end{itemize}

\begin{ex2}
We consider~\cref{ex:eq27}, where~$k=1$ and~$n=2$. We have the polynomials \[\begin{cases} E_1 := -(x_1-1)(u-1)+tu(2ux_1^2-uz_0+2uz_2-2x_1^2+u+x_1-2z_2-1),\\ E_2 := -x_2(u-1)+tu(2ux_1x_2+ux_1-2x_1x_2-x_1+x_2-z_2).\end{cases}\]
We compute~\begin{align*}
    \Det &= \det\begin{pmatrix} \partial_{x_1}E_1 & \partial_{x_2}E_1 \\ \partial_{x_1}E_2 & \partial_{x_2}E_2\end{pmatrix}
    = (4tu^2x_1-4tux_1+tu-u+1)(2tu^2x_1-2tux_1+tu-u+1).
\end{align*}
and the rather large polynomial $P = \det\begin{pmatrix} \partial_{x_1}E_1 & \partial_{u}E_1 \\ \partial_{x_1}E_2 & \partial_{u}E_2\end{pmatrix}$. We obtain by the duplication strategy the set~$\Sdup$ of polynomials given by
\begin{align*}
    \Sdup := \{&(E_1(x_1, x_2, u_1, z_0, z_1), E_2(x_1, x_2, u_1, z_0, z_1), \Det(x_1, x_2, u_1, z_0, z_1), P(x_1, x_2, u_1, z_0, z_1),\\
        &E_1(x_3, x_4, u_2, z_0, z_1), E_2(x_3, x_4, u_2, z_0, z_1), \Det(x_3, x_4, u_2, z_0, z_1), P(x_3, x_4, u_2, z_0, z_1)\}.
\end{align*}
Observe that~$\Sdup$ is built from~$8$ equations in the~$8$ unknowns~$x_1, x_2, x_3, x_4, u_1, u_2, z_0, z_1$ and the parameter~$t$.
\end{ex2}

Assuming that we did not deform~\eqref{eqn:initDDEsec4} using~\eqref{eqn:deformed_system}
(and thus that we did not introduce any deformation parameter~$\epsilon$),
it follows by construction that~$\Sdup$ is defined in the polynomial
ring~$\K(t)[x_1, \ldots, x_{n^2k}, u_1, \ldots, u_{nk}, z_0, \ldots, z_{nk-1}]$, and is built
from~$nk(n+2)$ equations in~$nk(n+2)$ unknowns ($t$ being considered as a parameter).
We introduce the extra
polynomial~$\operatorname{sat} := \prod_{i\neq j}{(u_i - u_j)(u_i-a)t}\in\K(t)[u_1, \ldots, u_{nk}]$
and define the saturated ideal~$\Idupinfty := \langle \Sdup \rangle : \operatorname{sat}^\infty$ in
$\K(t)[x_1, \ldots, x_{n^2k}, u_1, \ldots, u_{nk}, z_0, \ldots, z_{nk-1}]$. Recall that in practice, a generating set of the ideal~$\Idupinfty$ can be obtained by
computing a Gr\"obner basis of~$\langle \Sdup, m\cdot \operatorname{sat}-1\rangle
\cap\K(t)[x_1, \ldots, x_{n^2k}, u_1, \ldots, u_{nk}, z_0, \ldots, z_{nk-1}]$, where~$m$ is an extra variable introduced to remove the solution set of the equation~$\operatorname{sat}=0$ (see~\cite[p.~$205$, Thm.~$14$]{CoLiOSh15}]).~$\Idupinfty := \langle \Sdup \rangle : \operatorname{sat}^\infty$ in
$\K(t)[x_1, \ldots, x_{n^2k}, u_1, \ldots, u_{nk}, z_0, \ldots, z_{nk-1}]$.
In \eqref{hyp1} below we shall assume that~$\Idupinfty$ has expected dimension~$0$ over~$\K(t)$, without any multiplicity points, and we also assume that there exist~$nk$ distinct solutions to $\Det(u)=0$ in~$u$, all of them lying in~$\overline{\K}[[t^{\frac{1}{\star}}]]\setminus\overline{\K}$. All systems of DDEs that can be solved with the duplication approach satisfy this assumption, up to removing the multiplicity points (which is harmless for our objective of solving systems of~DDEs).
\begin{align}\label{hyp1}\tag{$\mathbf{H1}$}
\underline{\text{Hypothesis}~\mathbf{1}:} \nonumber&\bullet \;\Det(u)=0 \text{ admits } nk \text{ distinct solutions (in~$u$)} \text{ in } \overline{\K}[[t^{\frac{1}{\star}}]]\setminus \overline{\K}.\\
\nonumber& \bullet \;\text{The ideal~}\Idupinfty \text{ is radical of dimension}~0\text{ over~}\K(t).
\end{align}\vspace{-0.5cm}

\noindent Under~\eqref{hyp1}, we denote by~$U_1, \ldots, U_{nk}\in\overline{\K}[[t^{\frac{1}{\star}}]]\setminus \overline{\K}$ the distinct solutions to~$\Det(u)=0$.

%
Recall that after duplication of the variables~$x_1, \ldots, x_n, u$ in the initial set of
  polynomials~$(E_1, \ldots, E_n, \Det, P)$, we took the convention in~\cref{sec:proof_thm} to have
  for all~$1\leq i \leq n$, ~$1 \leq j \leq nk$ and all~$0\leq \ell \leq k-1$,
  the correspondence ``variables~$\leftrightarrow$~values'' as follows:
\begin{center}
  $x_{(j-1)n+i}\leftrightarrow F_i(t, U_j(t))$, \quad $u_j\leftrightarrow U_j(t)$, \quad 
  $z_{(i-1)k+\ell}\leftrightarrow (\partial_u^\ell F_i)(t, a)$. 
\end{center}
It follows from the second part of~\eqref{hyp1} that the elimination ideal~$\Idupinfty\cap\K[t, z_0]$ is not~$\{0\}$. In spirit of~\cite[Proposition~$2$]{BoNoSa23} and~\cite[Proposition~$2.1$]{BoChNoSa22}, we can show that
any nonzero element of~$\Idupinfty\cap\K[t, z_0]$ annihilates the series~$F_1(t, a)$.
\begin{lem}\label{prop:elimdup}
  Under~\eqref{hyp1}, any element $R\in\Idupinfty\cap\K[t, z_0]\setminus\{0\}$ satisfies
  $R(t, F_1(t, a))=0$.
\end{lem}
\begin{proof}
 First, it follows from~\eqref{hyp1} that the elimination ideal~$\Idupinfty\cap\K[t, z_0]$ is not~$\{0\}$.
 Using~$R\in\Idupinfty$ and the definition of the saturation, it follows that there exists some~$\ell\in\mathbb{N}$ such that $R\cdot(\prod_{i\neq j}(u_i - u_j)(u_i-a)t)^\ell$ can be written as a sum of multiples of the duplicated polynomials obtained from~$E_1, \ldots, E_n, \Det, P$. Specializing the resulting expression to the values~$x_{(j-1)n+i}= F_i(t, U_j(t))$, $u_j = U_j(t)$, $z_{(i-1)k+\ell} = (\partial_u^\ell F_i)(t, a)$ and using~\eqref{hyp1} implies that~$U_i(t)\neq U_j(t)$ and~$U_i(t)\neq a$ (whenever~$i\neq j$). This implies that~$R(t, F_1(t, a))=~0$.
\end{proof}

\begin{ex2}\label{ex17}
Consider~$\Sdup$, and compute a nonzero generator~$R\in\Q[t, z_0]$ of~$\langle \Sdup, m\cdot (u_1-u_2)\cdot (u_1-1)\cdot (u_2-1)t-1\rangle \cap\Q[t, z_0]$. For this computation,
which is not easy to perform with a naive approach, we rely on a recent work by the first author~\cite{Notarantonio23} which provides a Maple package with efficient implementations of algorithms for solving DDEs: one of these computes efficiently such an elimination polynomial. We obtain in a few seconds on a regular laptop that
\begin{align*}
    R \;=\; &(64t^3z_0^3+2t(24t^2-36t+1)z_0^2+(-15t^3+9t^2+19t-1)z_0+t^3+27t^2-19t+1)\\
    &\cdot (z_0-1)\cdot (2tz_0+t-1)\cdot (36-60z_0+t\cdot \widetilde{R}(t,z_0))
\end{align*}
annihilates~$F_1(t, 1)$; here we write $\widetilde{R}(t,z_0) \in \Q[t,z_0]$ for an explicit but rather big polynomial. Moreover, as~$F_1 = 1 + O(t)$ and~$F_1$ is not a constant, we can refine our conclusion and identify that the first factor of~$R$ is the minimal polynomial of~$F_1(t, 1)$. 
\end{ex2}

\subsubsection{Refined complexity and size estimates}
We prove a version of~\cref{thm:quantitative_estimates} in the case where the symbolic
deformation is not needed. In this context, we provide an upper bound on
the algebraicity degree of~$F_1(t, a)$ over~$\K(t)$. Also, we estimate the arithmetic 
complexity for the computation of a nonzero element of~$\Idupinfty\cap\K[t, z_0]$. 
\begin{prop}\label{cor:complexity}
  Let~$\K$ be a field of characteristic~$0$. Consider the system of
  DDEs~\eqref{eqn:initDDEsec4}, assume that~\eqref{hyp1} holds and denote by~$\delta$ an upper bound on the total degrees of~$f_1, \ldots, f_n, Q_1, \ldots, Q_n$. Then
  the algebraicity degree of~$F_1(t, a)$
  over~$\K(t)$ is bounded by~$n^{2nk}(\delta(k+1)+1)^{nk(n+2)}/(nk)!$.
  Moreover
  when~$\mathbb{K}$ is effective, there exists an algorithm computing a 
  nonzero polynomial~$R\in\K[t, z_0]$ such that $R(t, F_1(t, a))=0$,
  in~$\tilde{O}((nk\delta)^{5.26n(nk+k+1)})$ operations in~$\K$.
\end{prop}
\begin{proof}
  In spirit of the proof of~\cref{thm:quantitative_estimates}:

  \begin{itemize}
  \item Step 1: We compute an upper bound on the degree of the ideal~$\Idupinfty$ and then deduce an upper bound on the algebraicity degree of~$F_1(t, a)$.
  \item Step 2: We make explicit the number of variables and parameters in~$\Sdup$.
  \item Step 3: We describe and apply
    an algorithm which is itself based on the algorithm of~\cite[Theorem~$2$]{Schost03}
     in order to compute
    a nonzero annihilating polynomial of~$F_1(t, a)$.
    \item Step 4: Finally, we
      use the arithmetic complexity of the algorithm on which~\cite[Theorem~$2$]{Schost03} relies in order to estimate the complexity
      of our algorithm.
  \end{itemize}
  In order to define the polynomials~$E_1, \ldots, E_n \in \K(t)[x_1, \ldots, x_n, u, z_0, \ldots, z_{n k-1}]$
  associated to the numerator equations of~\eqref{eqn:initDDEsec4}, recall that it is necessary to multiply
  the~$i$th DDE in~\eqref{eqn:initDDEsec4} by a power~$m_i\in\mathbb{N}$ of~$(u-a)$. The polynomials~$Q_i$ being of total degree upper bounded by~$\delta$, and the DDEs considered being of order~$k$, it results that each~$m_i$ is upper bounded by~$k\delta$. So each of the~$E_i$'s has total degree upper bounded by~$d := \delta(k+1)+1$. Thus~$\Det$ and~$P$ have their respective total degrees bounded by~$nd$ as determinants of~$(n\times n)$-matrices whose entries are polynomials of total degrees bounded by~$d$. It remains to see that because of the radicality assumption in~\eqref{hyp1}, the Heintz-Bézout theorem~\cite[Theorem~$1$]{Heintz83} applies and implies that the degree of the ideal~$\Idupinfty$ is bounded by the product of the total degrees of the $nk(n+2)$ duplications~$E_1^{(1)}, \ldots, E_n^{(1)}, \Det^{(1)}, P^{(1)}$, $\ldots, E_1^{(nk)}, \ldots, E_n^{(nk)}, \Det^{(nk)}, P^{(nk)}$. Such a bound is given by~$\gamma(n, k, \delta) := n^{2nk}d^{nk(n+2)}$. Similarly to the proof of~\cref{thm:quantitative_estimates}, there is a group action of the symmetric group~$\mathfrak{S}_k$ over the zero set of~$\Idupinfty$ in~$\overline{\K(t)}^{nk(n+2)}$, obtained by permuting the duplicated blocks of coordinates, and which preserves the~$(z_0, \ldots, z_{nk-1})$-coordinate space. This group action implies that the algebraicity degree of~$F_1(t, a)$ is upper bounded
   by~$\gamma(n, k, \delta)/(nk)! =
   n^{2nk}(\delta(k+1)+1)^{nk(n+2)}/(nk)!$.
   
   Recall that we have one parameter~$t$, and~$nk(n+2)$ variables
  for the~$x_i$'s, the~$u_i$'s and the~$z_i$'s. We will make explicit some~$R\in\K[t, z_0]\setminus\{0\}$ annihilating~$F_1(t, a)$.
  To do so, we first compute (here we use the second part of~\eqref{hyp1} again) two polynomials~$V, W\in\K[t, \lambda]$
  such that: \textit{(i)}~$W$ is squarefree, \textit{(ii)} $\lambda$ is a new variable
  which is a $\K[t]$-linear
  combination of the~$nk(n+2)$ variables in~{Step~$2$}, \textit{(iii)} for
  all the zeros~$\mathbf{\alpha}\in\overline{\K(t)}^{nk(n+2)}$
  of~$\Sdup$ that are not solutions of~$\operatorname{sat}$,
  there exists~$\lambda_0\in\overline{\K(t, \epsilon)}$ solution in~$\lambda$
  of~$W(t, \lambda)=0$ such that
  $V(t, \lambda_0)/\partial_\lambda W(t, \lambda_0)$ is the $z_0$-coordinate of~$\mathbf{\alpha}$. Using these polynomials, it is straightforward
  to see by applying Stickelberger's theorem~\cite[Theorem~$1$]{Cox21}, using the radicality assumption in~\eqref{hyp1} and then applying~\cref{prop:elimdup}
  that the squarefree part of the
  resultant of~$z_0\cdot \partial_\lambda W(t, \lambda)- V(t, \lambda)$ and~$W(t, \lambda)$
  with respect to~$\lambda$ is a nonzero polynomial of~$\K[t, z_0]$ annihilating the series~$F_1(t, a)$.
  
   It remains to estimate the complexity of~{Step~$3$}.
  Following the proof of~\cref{thm:quantitative_estimates}, the application of the algorithm
  underlying~\cite[Theorem~$2$]{Schost03} allows us to compute~$V$ and~$W$
  in~$\tilde{O}((n^2kL+(n^2k)^4)\gamma(n, k, \delta)^2)$ operations in~$\K$, where~$L$ is the length of a
  straight-line program which evaluates the system~$\mathcal{S}_{{{\operatorname{dup}}}}$ and the
  polynomial~$\operatorname{sat}$. Since the
  cost for evaluating~$E_1, \ldots, E_n$ is included in~$O(d^{nk+n+2})$, then by the
  Baur-Strassen's theorem~\cite[Theorem~$1$]{BaSt83}, the complexity of evaluating the polynomials~$ \Det, P$
  is included in~$O(nd^{nk+n+2})$. Thus, evaluating~$E_1, \ldots, E_n, \Det, P$ has an arithmetic
  cost which is in~$O(n(\delta k)^{nk+n+2})$. As we considered~$nk$ duplication of the
  polynomials~$(E_1, \ldots, E_n,\Det, P)$, we find~$L\in O(n^3kd^{nk+n+2})\subset
  O(n^2k(\delta k)^{nk+n+2})$. It remains to compute the squarefree part~$R\in\K[t, z_0]\setminus\{0\}$ of
  the resultant of~$z_0\cdot \partial_\lambda W(t, \lambda) - V(t, \lambda)$ and~$W(t, \lambda)$
  with respect to~$\lambda$. Note that under~\eqref{hyp1}, the quantity~$\gamma(n, k, \delta)/(nk)!$
  bounds the partial degrees of~$R$.
  This resultant computation can be done
  by applying evaluation-interpolation with respect to~$t$ with~$O(n^{2nk}(\delta k)^{nk(n+2)}/(nk)!)$
  points. For each specialization at say~$t=\theta\in\mathbb{K}$, we apply for instance~\cite{Villard18}
  and deduce that the bivariate resultant computation of~$z_0\cdot \partial_\lambda W(\theta, \lambda)
  - V(\theta, \lambda)$ and~$W(\theta, \lambda)$ with respect to~$\lambda$ can be done
  in~\[\tilde{O}(\gamma(n, k, \delta)^{2.63}/(nk)!^{2.63})
  \subset
  \tilde{O}(n^{5.26nk}(\delta k)^{2.63 nk(n+2)}/(nk)!^{2.63})\] operations in~$\K$.
  Replacing~$L$ and~$\gamma(n, k, \delta)$ by their values and summing up all arithmetic costs,
  one deduces that the arithmetic cost for computing~$R$ is included in
    $\tilde{O}((nk\delta)^{5.26n(nk+k+1)})$.
\end{proof}

\begin{rmk}
  \textit{(i)} Note that the complexity in~\cref{cor:complexity} matches, as expected, the one that was proven in~\cite[Proposition~$4$]{BoNoSa23} in the case~$n=1$. One shall however see that the total degree of~$E_1$ is considered in~\cite[Proposition~$4$]{BoNoSa23}, whereas here we considered the total degree of~$Q_1$ and~$f_1$. Passing from the complexity in~\cref{cor:complexity} to the one in~\cite[Proposition~$4$]{BoNoSa23} is done by replacing~$\delta$ by~$\delta/k$.
  
  \textit{(ii)} It follows from the inclusion~$\tilde{O}((nk\delta)^{5.26n(nk+k+1)})\subset
  \tilde{O}((nk\delta)^{40(n^2k+1)})$
  that the arithmetic complexity stated in~\cref{cor:complexity} refines the arithmetic complexity
  stated in~\cref{thm:quantitative_estimates}.
\end{rmk}
As shown in~\cref{sec:generic_equations_of_first_order} and~\cref{sec:duplicate}, a natural way of solving~\eqref{eqn:initDDEsec4} is to compute a nonzero element $R\in\Idupinfty
\cap\mathbb{K}[t, z_0]$.
However, as already pointed out in~\cite[Section~$3$]{BoNoSa23} in the case~$n=1$, this process of duplicating variables yields an exponential growth of the degree of the ideal~$\Idupinfty$ with respect to the number~$nk$ of duplications~\cite[Proposition~$2$]{Heintz83}. It is known that the degree of the ideal generated by the polynomials in a given polynomial system is one of the main parameters controlling the complexity of solving the system (see, for example, \cite{BaFaSa15} for a complexity analysis of Faugère's $F_5$ algorithm in the context of homogeneous polynomials). Thus, if one can avoid the strategy of variable duplication one can potentially significantly reduce the exponent~$5.26n(nk+k+1)$ in the arithmetic complexity result of~\cref{cor:complexity}.


\subsection{Reducing a system of DDEs to a single functional equation}\label{sec:elim_duplicate}
\subsubsection{Main strategy}
In this subsection we elaborate on a strategy for solving systems of DDEs that avoids duplication of variables: it reduces the initial system to a single functional equation by eliminating the bivariate series $F_2, \ldots, F_n$ from the system $(E_1(u)=0, \ldots, E_n(u)=0)$. In a favourable situation, this reduction outputs a nonzero polynomial~$E\in\K(t)[x_1, u, z_0, \ldots, z_{nk}]$ such~that
\[
    E(u)\equiv E(F_1(t, u), u, F_1(t, a), \ldots, (\partial_u^{k-1}F_1)(t, a), \ldots, F_n(t, a), \ldots, (\partial_u^{k-1}F_n)(t, a)) = 0.
\]
The following lemma ensures that this strategy works under the following assumption: the saturated ideal~$\langle E_1, \ldots, E_n\rangle:\Det^\infty\cap\K(t)[x_1, u, z_0, \ldots, z_{nk-1}]$, denoted~by $\mathcal{J}$, is principal\footnote{Recall that an ideal is called principal if it is generated by only one element.} -- this natural assumption is observable on all examples of systems of DDEs we encountered so far.
\begin{lem}\label{lemE}
Consider~\eqref{eqn:init_system} and denote~$E_1, \ldots, E_n\in\K(t)[x_1, \ldots, x_n, u, z_0, \ldots, z_{nk-1}]$ the polynomials obtained after taking the numerators of~\eqref{eqn:init_system}. Assume that the ideal~$\mathcal{J}$ is principal, generated by some~$E\in\K(t)[x_1, u, z_0, \ldots, z_{nk-1}]$. Then~$E\neq 0$ and $E(u) = 0.$
\end{lem}
\begin{proof}
It results from the product of the diagonal elements in the matrix defining~$\Det$ that~$\Det \neq 0 \mod t$. Also,
the evaluation~$\Det(u)\in\K[[t]][[u]]$ is not equal to~$0$ (recall that~$\Det(u)$ is the specialization
of~$\Det$ to the series~$F_i(t, u)$ and~$(\partial_u^\ell F_i)(t, a)$). It follows from the Jacobian criterion
that the ideal
$\langle E_1, \ldots, E_n\rangle:\Det^\infty$ is radical and of dimension at most~$0$ over~$\K(t, z_0, \ldots, z_{nk-1})$. Thus~$E\neq 0$. Moreover, as~$E_1(u)=0, \ldots, E_n(u)=0$ and~$\Det(u)\neq 0$, it follows~$E(u)=0$. 
\end{proof}

\begin{ex2}
We compute a Gr\"obner basis of~$\langle E_1, E_2, m\cdot \Det -1 \rangle \cap \K(t)[x_1, u, z_0, z_1]$ and observe that it is not~$\{0\}$. Moreover, it is generated by a unique element 
\[E := -(x_1-1)(u-1) + tu(2ux_1^2-uz_0-2x_1^2+u+x_1-1) \in\K(t)[x_1, u, z_0, z_1].\]
\end{ex2}

For the remaining part of \cref{sec:elim_duplicate} we assume, as in~\cref{lemE}, that the ideal~$\mathcal{J}$ is principal. Thus the polynomial~$E$ obtained after eliminating~$x_2, \ldots, x_n$ satisfies~$E\neq 0$ and~$E(u)~=~0$.\\

A natural idea now is to use Bousquet-Mélou and Jehanne's method~\cite[Section~$2$]{BMJ06}. This reduces our problem either to solving a polynomial system with now $3nk+1$ unknowns and equations (where again the~$+1$ comes from the saturation polynomial~$\operatorname{sat}$), or to applying the more recent algorithm from~\cite[Section~$5$]{BoNoSa23}. Note that this method is not guaranteed to work in general because the functional equation given by~$E(u)=0$ may not be of a fixed point type for~$F_1$. As we will see, in order to make these approaches work, one can require the following~condition:
\begin{align}\label{C}\tag{\textbf{H2}}
  \underline{\text{Hypothesis~$\mathbf{2}$:}}  &\bullet \;\text{$\partial_{x_1} E(u)=0$ admits $nk$ distinct solutions (in~$u$) in $\overline{\K}[[t^{\frac{1}{\star}}]]\setminus\overline{\K}$,}\\
   \nonumber  &\bullet \;\text{The polynomial system 
    obtained after duplicating $nk$ times}\\
   \nonumber  &\;\;\;\;\;\text{the variables~$x_1, u$ in~$(E, \partial_{x_1}E, \partial_uE)$ and adding~$\operatorname{sat} = 0$}\\
   \nonumber & \;\;\;\;\;\text{induces
     an ideal of dimension~$0$ over~$\K(t)$.}
\end{align}
We emphasize that we are not aware of any system of DDEs inducing some~$E$ that is not
of a fixed-point type in~$F_1$: when the system is generic (that is, for generic choices of $f_i$'s and $Q_i$'s) we could observe by generating lots of examples that this fixed-point nature is indeed satisfied.

This resolution strategy being different from the duplication of variables approach from~\cref{sec:duplicate}, we investigate below how their outputs compare.

\subsubsection{Theoretical comparison with the strategy from~\cref{sec:duplicate}}

The following proposition ensures that 
the first part of~\eqref{hyp1} implies the first part of~\eqref{C}. Recall that $\mathcal{J} \coloneqq \langle E_1, \ldots, E_n\rangle:\Det^\infty\cap\K(t)[x_1, u, z_0, \ldots, z_{nk-1}]$.
\begin{prop}\label{prop:preservation_Ui}
    Let~$U(t)\in\overline{\K}[[t^{\frac{1}{\star}}]]$ be a solution of~$\Det(u)=0$ such that none of the~$(n-1)\times(n-1)$ minors of~$(\partial_{x_j}E_i(U(t)))_{1\leq i, j\leq n}$ is zero.
    Consider $E\in\mathcal{J}$. Then~$U(t)$ is
    also a solution in~$u$ of the equation~$\partial_{x_1}E(u) = 0$.
\end{prop}
\begin{proof}
Since $E\in\mathcal{J}$, there exist polynomials $V_1,\dots,V_n \in \K(t)[x_1,\dots,x_n, u, z_0, \ldots, z_{nk-1}]$ such that $E(U) = \sum_{i=1}^n E_i(U) V_i(U)$. Differentiating with respect to $x_j$ for $j=1,\dots,n$ and using that $E_i(U) = 0$ and that $E$ does not depend on $x_j$ for~$j \geq 2$, we find \begin{equation}\label{eq:R_matrix_V}
    \begin{pmatrix}
    \partial_{x_1} E(U)\\
    0\\
    \vdots\\
    0
    \end{pmatrix} = 
    \begin{pmatrix}
    \partial_{x_1}E_1(U) & \dots & \partial_{x_1} E_n(U) \\
    \vdots & \ddots & \vdots \\
    \partial_{x_n}E_1(U) & \dots & \partial_{x_n} E_n(U)
    \end{pmatrix}
    \begin{pmatrix}
    V_1(U)\\
    \vdots\\
    V_n(U)
\end{pmatrix}.
\end{equation}
By definition of $U$, the matrix $(\partial_{x_j }E_i(U))_{i,j}$ is singular and each of its $(n-1)\times(n-1)$ minors is nonzero. It follows that we can express the first row of the matrix as a linear combination of the other rows, then (\ref{eq:R_matrix_V}) implies that $\partial_{x_1} E(U) = 0$.
\end{proof}

It remains to understand the second part of~\eqref{C}.
We reuse the notations from the proof of~\cref{prop:preservation_Ui} by writing~$E = \sum_{i=1}^n{E_i\cdot V_i}$
for~$E$
a generator of $\mathcal{J}$ and for~$V_1, \ldots, V_n\in\K(t)[x_1,\dots,x_n, u, z_0, \ldots, z_{nk-1}]$.
Also, we define the new geometric assumption~\eqref{hypP}:
\begin{align}\label{hypP}\tag{\textbf{P}}
\underline{\text{Hypothesis}~\mathbf{(P)}:} &\;\; \text{Every point~$\mathbf{\beta}\in\acfield{t}^{nk+2}$ on which~$E$ vanishes is the projection onto}  \\   \nonumber&\;\;\text{the~$\{x_1, u, z_0, \ldots, z_{nk-1}\}$-coordinate space of a point $\alpha\in\acfield{t}^{nk+n+1}$ }\\
     \nonumber&\;\;\text{vanishing simultaneously~$E_1, \ldots, E_{n-1}$ and~$E_n$.}
\end{align}
Observe that the zero set of~$E$ in~$\acfield{t}^{nk+2}$ is, by the closure theorem~\cite[Theorem~$3$, p.$131$]{CoLiOSh15}, the Zariski closure in~$\acfield{t}^{nk+2}$ of the projection of the solution set in~$\acfield{t}^{nk+n+1}$ of~$\langle E_1, \ldots, E_n\rangle:\Det^\infty$ onto the~$\{x_1, u, z_0,\ldots, z_{nk-1}\}$-coordinate space. Assumption~\eqref{hypP}
formulates that the boundary of the projection belongs to the projection itself. 

From now on we will denote by~$\underline{x}$ (resp.~$\underline{z}$) the variables~$x_1, \ldots, x_n$ (resp. $z_0, \ldots, z_{nk-1}$). 
\begin{rmk}
 Once a Gr\"obner basis~$G$ of the ideal~$\langle E_1, \ldots, E_n\rangle:\Det^\infty$ is computed for the block order~$\{x_2, \ldots, x_n\}\succ_{lex} \{x_1, u, \underline{z}\}$, the set of points~$\beta$ in~\eqref{hypP} can
be characterized by the extension theorem~\cite[Theorem~3, Chap~3.1]{CoLiOSh15} as the solutions of explicit polynomial equations and inequations built from~$G$: namely $\beta$ are those solutions of~$G\cap\K(t)[x_1, u, \underline{z}]$ that do not vanish simultaneously all the leading terms of~$G\setminus (G\cap\K(t)[x_1, u, \underline{z}])$ for the degrevlex monomial order~$\{x_2, \ldots, x_n\}$ (the variables~$t, x_1, u, \underline{z}$ shall be seen as parameters once~$G$ is computed). These leading coefficients are thus polynomials of~$\K(t)[x_1, u, \underline{z}]$.
\end{rmk}
We formulate the below statement, whose purpose is to interpret geometrically the link between the algebraic sets~$V(E,\partial_{x_1}E, \partial_uE)\subset\acfield{t}^{nk+2}$ 
and~$V(E_1, \ldots, E_n, \Det, P)\subset\acfield{t}^{nk+n+1}$.
\begin{prop}\label{prop:preservation_dim}
Assume that~$\mathcal{J}$ is principal.
Let~$E\in\mathbb{K}(t)[x_1, u, z_0, \ldots, z_{nk-1}]$ be a generator of~$\mathcal{J}$ 
and suppose that~$V(E)\subset\overline{\K(t)}^{nk+2}$ is smooth.
Assume, moreover, that~\eqref{hypP} holds and denote
by $\pi_{x_1, u, \underline{z}}:(\underline{x}, u, \underline{z})
\in\acfield{t}^{nk+n+1}\mapsto(x_1, u, \underline{z})\in\acfield{t}^{nk+2}$
the canonical projection onto the~$(x_1, u, \underline{z})$-coordinate space.
Then we have the inclusion~$V(E, \partial_{x_1}E, \partial_uE)\subset
\pi_{x_1, u, \underline{z}}(V(E_1, \ldots, E_n, \Det, P))$.
\end{prop}
\begin{proof}
Writing~$E = \sum_{i=1}^nE_i\cdot V_i$ for
some~$V_1, \ldots, V_n\in\K(t)[\underline{x}, u, \underline{z}]$
and considering the derivatives~$\partial_{x_i}E$ for~$i=1, \ldots, n$, we find by the Leibniz rule
\begin{equation}\label{eq:R_matrix_V1}
    \begin{pmatrix}
    \partial_{x_1} E\\
    0\\
    \vdots\\
    0
    \end{pmatrix} = 
    \begin{pmatrix}
    \partial_{x_1}E_1 & \dots & \partial_{x_1} E_n \\
    \vdots & \ddots & \vdots \\
    \partial_{x_n}E_1 & \dots & \partial_{x_n} E_n
    \end{pmatrix}
    \begin{pmatrix}
    V_1\\
    \vdots\\
    V_n
\end{pmatrix} + 
    \begin{pmatrix}
    \partial_{x_1}V_1 & \dots & \partial_{x_1} V_n \\
    \vdots & \ddots & \vdots \\
    \partial_{x_n}V_1 & \dots & \partial_{x_n} V_n
    \end{pmatrix}
    \begin{pmatrix}
    E_1\\
    \vdots\\
    E_n
\end{pmatrix}.
\end{equation}
Now for any~$\beta\in V(E, \partial_{x_1}E)$ and for
all~$\alpha\in\pi_{x_1, u, \underline{z}}^{-1}(\beta)\cap V(E_1, \ldots, E_n)$, we obtain that 
\begin{equation*}
    \begin{pmatrix}
    0\\
    \vdots\\
    0
    \end{pmatrix} = 
    \begin{pmatrix}
    \partial_{x_1}E_1(\alpha) & \dots & \partial_{x_1} E_n(\alpha) \\
    \vdots & \ddots & \vdots \\
    \partial_{x_n}E_1(\alpha)& \dots & \partial_{x_n} E_n(\alpha)
    \end{pmatrix}
    \begin{pmatrix}
    V_1(\alpha)\\
    \vdots\\
    V_n(\alpha)
\end{pmatrix}.
\end{equation*}
Thus it suffices to prove that~$V_i(\alpha)\neq0$ (for some~$i\in\{1, \ldots, n\}$) in order to to see that~$\Det(\alpha)=0$. As~$V(E)\in\acfield{t}^{nk+2}$ is smooth, the vector of partial derivatives of~$E$ w.r.t the variables $x_1, u, z_0, \ldots, z_{nk-1}$ does not vanish identically at~$\beta$. As~$E = \sum_{i=1}^nE_i\cdot V_i$, this enforces one of the~$V_i$ to be nonzero at~$\alpha$: else by writing all the partial derivatives of~$E$ with the Leibniz rule applied to the sum expression~$E = \sum_{i=1}^nE_i\cdot V_i$ and using~$\alpha\in V(E_1, \ldots, E_n, V_1, \ldots, V_n)$ would imply that
the vector of partial derivatives of~$E$ with respect to~$x_1, u, z_0, \ldots, z_{nk-1}$
vanishes identically~at~$\beta$.

Now consider the vector of derivatives~$(\partial_uE, \partial_{x_2}E, \ldots, \partial_{x_n}E)$ and write it in the form~\eqref{eq:R_matrix_V1}. This allows to reuse the above argument to show that~$P(\alpha)=0$, for any~$\alpha\in\acfield{t}^{nk+n+1}$ vanishing simultaneously~$\partial_u E, E_1, \ldots, E_n$. It follows that under assumption~\eqref{hypP},
we have the claimed inclusion~$V(E, \partial_{x_1}E, \partial_uE)\subset \pi_{x_1, u, \underline{z}}(V(E_1, \ldots, E_n, \Det, P))$.
\end{proof}

Let us mention that under the technical assumptions introduced in~\cref{prop:preservation_Ui} and~\cref{prop:preservation_dim}: if~$\Det(u)=0$ admits~$\ell$ distinct solutions (in~$u$) in~$\overline{\K}[[t^{\frac{1}{\star}}]]$ then so does~$\partial_{x_1}E(u)=0$.

Moreover still under the assumptions of~\cref{prop:preservation_Ui} and~\cref{prop:preservation_dim}, if the ideal~$\Idupinfty$
of~\cref{subsec:general_strategy_duplicate}
has dimension~$0$ over~$\K(t)$, then the application of~\cite[Section~$5$]{BoNoSa23} to the functional equation~$E(u)=0$ outputs a nonzero polynomial~$R\in\K[t, z_0]$ such that~$R(t, F_1(t, a))=0$.
 
 We gather below assumptions that we will use afterwards. Also for any multivariate polynomial~$p$, we denote by~$\operatorname{SqFreePart}(p)$ the squarefree part of~$p$.
 \begin{align}\label{hyp3}\tag{\textbf{H3}}
     \underline{\text{Hypothesis $3$:}} &\bullet \text{\eqref{hyp1} and~\eqref{hypP} hold},\\
          \nonumber                          & \bullet\text{$\mathcal{J}$ is principal
          and generates a smooth algebraic set in~$\overline{\K(t)}^{nk+2}$},\\
   \nonumber &\bullet \text{For all~$U(t)\in\overline{\K}[[t^{\frac{1}{\star}}]]$ solution of~$\Det(u)=0$,
    the~$(n-1)\times (n-1)$ minors}\\
    \nonumber& \;\;\;\;\text{of~$(\partial_{x_j}E_i(U(t)))_{1\leq i, j\leq n}$ are nonzero. }
 \end{align}
 \begin{prop}\label{prop:division}
    Assume that~\eqref{hyp3} holds. Denote~$R_1\in\K[t, z_0]$ a generator of~$\Idupinfty\cap\K[t, z_0]$ and~$R_2\in\K[t, z_0]$ the polynomial computed by the duplication approach\footnote{This duplication approach should in practice be replaced by the more efficient algorithm from~\cite[Section~$5$]{BoNoSa23}.} applied to the functional equation~$E(u)~=~0$.
    Then~$\operatorname{SqFreePart}(R_2)$ divides~$\operatorname{SqFreePart}(R_1)$.
 \end{prop}
\begin{proof}
The statement is a direct consequence of~\cref{prop:preservation_Ui} and \cref{prop:preservation_dim}.
\end{proof}
In this section, so far, we compared two different strategies based for the first one on a duplication of variables argument, and for the second one on the reduction to a single equation. We showed that up to considering squarefree parts, the output polynomial of the second strategy divides, under technical assumptions, the output polynomial of the first strategy. 

Even if we strongly believe that the introduced assumptions are reasonable for practical considerations, we shall however underline that the literature contains some degenerate examples that go against this believe. The following example illustrates this degeneracy and thus strengthen the relevancy of the strategy from~\cref{sec:duplicate}, even if the arithmetic complexity in~\cref{cor:complexity} can seem despairing. 

\begin{ex2}\label{ex18}
In this case~\cref{prop:preservation_Ui} cannot be applied, since 
\begin{align*}
    \Det = \det\begin{pmatrix} 4tu^2x_1-4tux_1+tu-u+1 & 0 \\\star& 2tu^2x_1-2tux_1+tu-u+1\end{pmatrix},
\end{align*}
so any of the two distinct solutions~$U_1, U_2\in\overline{\K}[[t^{\frac{1}{\star}}]]$ to the equation~$\Det(u)=0$ annihilates at least two 
coefficients of the matrix~$(\partial_{x_j}E_i)_{1\leq i, j\leq 2}$. Moreover, 
it is straightforward to see that for the polynomial~$E = -(x_1-1)(u-1) + tu(2ux_1^2-uz_0-2x_1^2+u+x_1-1) \in\K(t)[x_1, u, z_0, z_1]$ the equation $\partial_{x_1}E(u)=0$ has only one solution (in~$u$) in the ring~$\overline{\K}[[t^{\frac{1}{\star}}]]$ while~$E(u)=0$ involves two univariate series~$F_1(t, 1)$ and~$F_2(t, 1)$. Thus, on this example, the strategy described in the current section fails.
\end{ex2}

\section{Conclusion and future works} \label{sec:discussions}

We proved in \cref{thm:main_thm} that solutions of systems of discrete differential equations are algebraic functions. The proof uses the observation that, up to a symbolic deformation pushing all degeneracy of a given system of the form~\eqref{eqn:init_system} to some higher powers of~$t$, all the computations become explicit. In addition, we obtained in \cref{thm:quantitative_estimates} quantitative estimates for the size of an annihilating polynomial for any solution~$F_i$, together with arithmetic complexity estimates. 

We also refined the complexity estimate in~\cref{cor:complexity} and also obtained an algebraicity bound for all the~$F_i(t, a)$. Moreover, we started an algorithmic analysis of the problem of solving systems of DDEs. In this first work we compared two natural strategies for solving systems of~DDEs. We identified rigorous assumptions ensuring their applicability illustrating them on the Example \ref{ex:eq27}, where the reduction to a single equation fails, while the duplication approach works. Under these assumptions, the outputs of these two strategies are linked via \cref{prop:division}.

Our paper contains several research directions which we plan to investigate in the future:

\paragraph{Algorithmic part:} This article set the foundations of an algorithmic work that must be conducted in order to understand better the many subtle and degenerate situations that one 
can encounter with systems of DDEs from the literature. 
For instance, the following points should be addressed:
\begin{itemize}
    \item \textbf{Validity of the assumptions:} It seems clear that up to deforming symbolically as in~\eqref{eqn:deformed_system}, assumption~\eqref{hyp1} can be assumed to hold. Still, we do not know at this stage if~\eqref{hypP} holds once~\eqref{eqn:init_system} is deformed as in~\eqref{eqn:deformed_system}. Also, we would like to have results saying that the all introduced assumptions hold for generic choices of systems of DDEs.
    \item \textbf{New algorithms for solving systems of DDEs:} In~\cite[Section~$5$]{BoNoSa23} an algorithm was designed that computes in case of scalar DDEs the same output as the algorithm described in~\cref{sec:duplicate} and that avoids the duplication of variables. This strategy can be naturally extended to the context of systems of DDEs. Once this is done, a first task would be to compare this new strategy with the ones described in~\cref{sec:elim_duplicate}. Also, the case of systems of linear DDEs contains many applications in which it is possible to use the linearity in order to speed up the computations. All this should be clarified and studied in depth. 
   \end{itemize}

\paragraph{Theoretical part:} In the case of more than one catalytic variable, the algebraicity of the solutions remains guaranteed by Popescu's theorem~\cite{Popescu85a} under the additional assumption that the variables are ``nested''. It would be interesting to reprove this result with elementary tools in the case of fixed-point equations and to introduce new algorithms for solving these equations.

\paragraph{Acknowledgments:}
We thank the anonymous referees for their helpful comments and suggestions. 
We would also like to thank Mohab Safey El Din for many interesting discussions and Alin Bostan for his great support and valuable comments on an early version of this manuscript. We are also grateful to Mireille Bousquet-Mélou for exciting and enlightening discussions, and to Eva-Maria Hainzl and Wenjie Fang for pointing out to us the existence of the periodic Tamari intervals and their associated DDEs. Also, we are grateful to Mathilde Bouvel and Nicholas Beaton for pointing out~\cite{BBGR19} for applications of our paper. Finally, we are thankful to the anonymous reviewers for the helpful suggestions and corrections. 

Both authors were supported by the ANR-19-CE40-0018 \href{https://specfun.inria.fr/chyzak/DeRerumNatura/}{De Rerum Natura} project and the French–Austrian ANR-22-CE91-0007 \& FWF I6130-N project EAGLES; the second author was funded by \href{https://www.artech.at/}{A\&R TECH} as well as the DOC fellowship P-26101 of the \href{https://www.oeaw.ac.at/en/1/austrian-academy-of-sciences}{ÖAW}, the \href{https://oead.at/en/}{WTZ collaboration project} FR-09/2021 \& FR 02/2024 and the project P34765 of the \href{https://www.fwf.ac.at/en}{ Austrian science fund FWF}.

\newpage

{\footnotesize \bibliographystyle{abbrv}
\bibliography{sample}
}

\end{document}